\documentclass[a4paper,leqno]{amsart}
\usepackage{graphicx, amsmath, amssymb, amsthm, hyperref, verbatim, multicol, mathrsfs, tikz, caption, subcaption,float}
\usepackage{stmaryrd}
\usepackage{mathtools}  
\usepackage{cleveref}   
\usepackage{todonotes}
\usepackage{tikz-cd}
\usepackage{physics}
\usepackage{enumitem}   
\usepackage{wrapfig, blindtext}
\newtheorem{thm}{Theorem}[section]
\newtheorem{lemm}[thm]{Lemma}

\newtheorem{rem}[thm]{Remark}

\newtheorem{prop}[thm]{Proposition}

\theoremstyle{definition}
\newtheorem{defi}[thm]{Definition}

\theoremstyle{remark}

\DeclareMathOperator{\diam}{diam}
\DeclareMathOperator{\Int}{int}

\makeatletter

\newcommand{\Rom}[1]{\expandafter\@slowromancap\romannumeral #1@}

\begin{document}

\title[Triangulating metric surfaces]{Triangulating metric surfaces}

\author{Paul Creutz}
	\address{Department of Mathematics, University of Cologne, Weyertal 86-90, 50931 K\"oln, Germany.}
	\email{pcreutz@math.uni-koeln.de}

\author{Matthew Romney}
	\address{Mathematics Department, Stony Brook University, Stony Brook NY, 11794, USA.}
	\email{matthew.romney@stonybrook.edu}

\thanks{Both authors were partially supported by DFG-grant SPP 2026, and the first-named author also by DFG-grant SFB/TRR 191 ``Symplectic structures in Geometry, Algebra
and Dynamics." \newline {\it 2010 Mathematics Subject Classification.} Primary 53C45. Secondary 52A10, 53A05.
	}
\maketitle
\begin{abstract}
We prove that any length metric space homeomorphic to a surface may be decomposed into non-overlapping convex triangles of arbitrarily small diameter. This generalizes a previous result of Alexandrov--Zalgaller for surfaces of bounded curvature.   
\end{abstract} 

\section{Introduction}
\subsection{Main results}
The theory of surfaces of bounded curvature was developed beginning in the~1940s as a generalization of two-dimensional Riemannian geometry.  One of the central results of this theory is that any surface of bounded curvature is the limit of two-dimensional Riemannian manifolds of uniformly bounded integral curvature. A key step in the proof of this result is to show that every surface of bounded curvature admits a triangulation by convex geodesic triangles of arbitrarily small diameter. Although versions of the approximation and triangulation theorems appeared earlier in works of Alexandrov~\cite{Ale:48} and Zalgaller~\cite{Zal:56}, complete proofs have only been published in their monograph~\cite{AZ:67}.  We also refer the reader to surveys by Fillastre \cite{FS:20}, Reshetnyak \cite{Res:93} and Troyanov \cite{Tro:09} for an overview of the subject.

 While surfaces of bounded curvature remain an active research topic (see for instance \cite{AGW:10,AB:16,BB:04, BE:00,BL:03,Deb:20,FS:19,Kok:14}), various classes of metric surfaces that do not fall into this setting have also been widely studied in recent years. These include reversible Finsler surfaces~\cite{BL:10,Bry:06,BI:02,DPMMS:20}, minimal surfaces in spaces satisfying a quadratic isoperimetric inequality \cite{CR:20,LW:18a,LW:20}, metric minimizing disks \cite{PS:19}, Ahlfors 2-regular quasispheres~\cite{BK:02}, quasiconformal images of planar domains \cite{Raj:17}, and fractal spheres~\cite{BM:17}.

In this paper, we generalize the theorem of Alexandrov--Zalgaller on the existence of triangulations (see Theorem III.2 in \cite{AZ:67}) to the case of arbitrary geodesic surfaces. In its simplest version, our result is the following.
\begin{thm}
\label{thm:main}
Let $X$ be a geodesic metric space homeomorphic to a closed surface and $\varepsilon >0$. Then $X$ may be decomposed into finitely many non-overlapping convex triangles, each of diameter at most $\varepsilon$.
\end{thm} 
 Here, a \emph{triangle} is a subset of~$X$ homeomorphic to the closed disk whose boundary is the union of three geodesics. We remark that, like the corresponding result for surfaces of bounded curvature \cite[Thm. III.2, p.59]{AZ:67}, our \Cref{thm:main} does not give a triangulation of~$X$ in the classical sense. The difference is that we do not require adjacent triangles to intersect along entire edges.

One step in the original proof by Alexandrov--Zalgaller of the existence of triangulations is to show that any point in a surface of bounded curvature can be enclosed by an arbitrarily short piecewise geodesic curve. As noted in \cite{AZ:67} and \cite{Res:93}, this is the only step that relies on the assumption of bounded curvature. Simple examples show that this property does not hold for general geodesic surfaces, and indeed it is difficult to prove even for surfaces of bounded curvature using the definition in~\cite{AZ:67}; see~\cite[Sec. III.5]{AZ:67} and \cite[p.81]{Res:93}. Instead, we give a relatively short argument showing that every point has a neighborhood that may be covered by finitely many polygons, each of arbitrarily small perimeter; see \Cref{lemm3}. Thus our approach also simplifies the original proof even in the bounded curvature case. 

In principle, except for this difference, we are able to follow the proof given in \cite{AZ:67} for surfaces of bounded curvature. However, this proof contains several technical errors. These errors are related to the fact that geodesics at the present level of generality can be highly non-unique and hence intersect in complicated ways. To handle this issue, Alexandrov--Zalgaller consistently use the notion of what they call \emph{systems of geodesics without superfluous intersections}. It turns out that the principle they use to pass to such systems is not valid in general, even in the bounded curvature setting; see the discussion in \Cref{sec:sup}. Since this principle is applied at numerous places, we choose to give a complete self-contained proof of \Cref{thm:main}. In particular, Lemma III.6 in \cite{AZ:67} is not correct, and this portion of the proof requires a more refined approach; see Sections~\ref{subsec:covsup} and~\ref{subsec:covover}.

In the bounded curvature setting, the proof of approximation by Riemannian 2-manifolds requires additional technical conclusions beyond those given in \Cref{thm:main}; compare \cite[Thm. III.3, p.61]{AZ:67}. Our proof equally allows for these conclusions, and thus we now give the following general version of our main result. 

\begin{thm}
\label{thm:maingen}
Let $X$ be a length metric space homeomorphic to a surface such that every boundary component of $X$ is a piecewise geodesic curve, and let $\varepsilon >0$. Then $X$ may be covered by a locally finite collection of non-overlapping triangles $(T_i)_{i \in \mathcal{I}}$ such that the following hold for each $i \in \mathcal{I}$.
\begin{enumerate}[label=(\roman*)]
    \item The triangle $T_i$ is convex relative to its boundary. \label{item:thm_i}
       \item The diameter of $T_i$ is at most $\varepsilon$. \label{item:thm_ii}
    \item The triangle $T_i$ is non-degenerate. \label{item:non_degenerate}
    \item $\partial T_i \setminus \partial X$ consists of transit points. \label{item:thm_v}
\end{enumerate}
\end{thm}
Here, a triangle $T$ is called \emph{non-degenerate} if the corresponding Euclidean comparison triangle is non-degenerate, or equivalently if all triangle inequalities for the sides of $T$ are strict. See Sections~\ref{sec:prelis} and \ref{sec:boundary_convexity} for definitions of the other terms used in the statement of Theorem~\ref{thm:maingen}. Both in our paper and in \cite{AZ:67}, the conclusion \ref{item:non_degenerate} that the triangles $T_i$ are non-degenerate can be achieved \emph{a posteriori} by showing that any degenerate triangle is decomposable into non-degenerate triangles; see \cite[Lem.~III.7, p.60]{AZ:67} and \Cref{lemm:deg_triang} below. However, the proof of this fact in \cite{AZ:67} relies heavily on the assumption of bounded curvature, and hence an original argument is needed to obtain \Cref{lemm:deg_triang}. This is the only step where the proof for surfaces of bounded curvature turns out to be much simpler than the general case. 

Note that \Cref{thm:maingen} also applies to non-compact surfaces, and that in this case a localized version of conclusion \ref{item:thm_ii} is possible. See \Cref{rem:nc} below. The conclusion \ref{item:thm_v} about transit points is not included in the statement of \cite[Thm~III.2]{AZ:67}, although it is mentioned immediately after. This property is crucial for the proof of the approximation theorem; compare~\cite[Thm.~II.11, p.47]{AZ:67}, \cite[p.65]{AZ:67} and \cite[p.86]{Res:93}.

Recall that the original motivation for \Cref{thm:maingen} in the bounded curvature setting was to prove that any such surface is the limit of two-dimensional Riemannian manifolds of uniformly bounded integral curvature. In \cite{NR:21}, Ntalampekos and the second author apply \Cref{thm:maingen} to prove an analogous approximation theorem for general length metric surfaces of locally finite Hausdorff $2$-measure, without assuming the bounded curvature condition. This generalized approximation has further applications concerning the \textit{uniformization problem} for metric surfaces, which asks for the existence of geometrically well-behaved parametrizations of metric surfaces. Previous results of this type require geometric assumptions on the surface such as linear local connectivity, the validity of a quadratic isoperimetric inequality, or bounds on conformal modulus of curve families \cite{BK:02, LW:20, Raj:17}. In turn, making use of the approximation theorem, a very general quasiconformal parametrization result for length surfaces is derived in \cite{NR:21} that only requires the surface to have locally finite Hausdorff $2$-measure.

\subsection{Organization and outline of proof}
 We first recall in \Cref{sec:prelis} the basic notions required from metric geometry. Convexity relative to the boundary and its role in the proof of our main result are discussed in \Cref{sec:boundary_convexity}. 
 In \Cref{sec:sup}, we discuss the methods needed for handling superfluous intersections. We then prove \Cref{thm:maingen} in~\Cref{sec:proof}, with the exception of the non-degeneracy conclusion \ref{item:non_degenerate}. Finally, in \Cref{sec:degenerate_triangles}, we verify that one can further subdivide a triangulation so that the non-degeneracy conclusion is satisfied. 
 
 The proof of \Cref{thm:maingen} consists of several steps. First, a relatively simple argument, given in \Cref{subsec:diamcov}, shows that $X$ is covered by polygons of small diameter. Next, in \Cref{subsec:covper}, we improve this to a cover by polygons having both small diameter and small perimeter. This is the main step where our proof differs from, and simplifies, the classical proof for surfaces of bounded curvature. In \Cref{subsec:covabs}, we use an argument from \cite{AZ:67} to find a cover by small polygons that are also absolutely convex. The main remaining difficulty is to show that one can pass to a cover by polygons that are also non-overlapping. In \cite{AZ:67}, this is achieved by Lemma~III.6. The proof of this lemma, unfortunately, is not correct. As a replacement, we use two intermediate steps. First, in \Cref{subsec:covsup}, we show that we can pass to a cover by boundary convex polygons such that the boundary edges form a locally finite graph. We then show in \Cref{subsec:covover} that, from such a cover, we may pass to one that consists of non-overlapping boundary convex polygons. As a final step, it suffices to show that every boundary convex polygon may be cut into finitely many non-overlapping boundary convex triangles. This relatively simple argument is provided in \Cref{subsec:covtria}.
 
Two important tools for manipulating polygons are given in Lemmas~\ref{lemm:intersect} and~\ref{lemm:part} below. These show that boundary convexity, unlike other notions of convexity, is stable with respect to certain operations of both intersection and subdivision. Thus, to prove Theorem~\ref{thm:main}, it seems necessary to work with boundary convex subsets. However, the definition of a boundary convex subset relies heavily on the assumption that the ambient space $X$ is a surface and hence that locally the Jordan curve theorem applies. In particular, it does not seem straightforward to generalize Theorem~\ref{thm:main} from surfaces to even more general spaces such as two-dimensional simplicial complexes. 
 
\subsection*{Acknowledgments}
We thank Alexander Lytchak for encouraging us to work on this topic and for his great support. We also thank François Fillastre, Mikhail Katz, Christian Lange, Dimitrios Ntalampekos, Raanan Schul and Stephan Stadler for helpful comments and suggestions that improved the presentation of this article. Finally, we thank the referees for carefully reading the paper and useful feedback.

\section{Preliminaries}
\label{sec:prelis}

\subsection{Metric geometry}
We first review the relevant definitions from metric geometry. Let $(X,d)$ be a metric space. Recall that a function $d \colon X \times X \to [0, \infty)$ is a metric if it is positive definite and symmetric and satisfies the triangle inequality.  
For each pair of subsets $A,B\subset X$ and $x\in X$, let $d(A,B)=\inf_{a\in A,b \in B} d(a,b)$ and $d(x,A)=d(\{x\},A)$. The \emph{diameter} of $A$ is defined by $\diam(A)=\sup_{a,a'\in A} d(a,a')$. A family~$(A_i)_{i\in \mathcal{I}}$ of subsets of~$X$ is \emph{locally finite} if every compact set $K\subset X$ intersects at most finitely many of the sets~$A_i$.

A \textit{curve} is a continuous map $\gamma\colon I\to X$, where $I\subset \mathbb{R}$ is an interval.  We denote the image of the curve $\gamma$ by $|\gamma |$. A curve $\gamma\colon I \to X$  is \emph{compact} if $I=[a,b]$ is compact. A continuous map $h\colon [a,b]\times[0,1]$ is called a \emph{path homotopy} from $\gamma_0=h(\cdot,0)$ to $\gamma_1=h(\cdot,1)$ if it is constant on $\{a\}\times [0,1]$ and $\{b\}\times[0,1]$. A compact curve~$\gamma$ is \emph{closed} if both its endpoints coincide and \emph{simple} if it does not have self-intersections, except possibly coinciding endpoints. A simple closed curve is called a \emph{Jordan curve}, and a simple non-closed curve is called an \emph{arc}. The length of a curve~$\gamma$ is denoted by~$\ell(\gamma)$. Note that in general it is possible that $\ell(\gamma)=\infty$. The concatenation of two curves $\gamma_1, \gamma_2$ is denoted by $\gamma_1 * \gamma_2$. The reverse of the curve $\gamma$ is denoted by $\bar{\gamma}$.
 
 The metric space $X$ is a \emph{length space} if $d(x,y)=\inf_\gamma \ell(\gamma)$ for all $x,y\in X$, where the infimum is taken over all compact curves $\gamma$ joining $x$ to $y$. The space $X$ is a \emph{geodesic space} if additionally this infimum is attained for all pairs of points~$x,y$. A compact curve is a \emph{geodesic} if its length equals the distance between its endpoints. A compact curve is \emph{piecewise geodesic} if it is the concatenation of finitely many geodesics. A non-compact curve is \emph{piecewise geodesic} if its restriction to each compact interval is piecewise geodesic. 
 
 A \textit{surface} is a topological 2-manifold with boundary. A surface is \textit{closed} if it is compact and its boundary is empty. We recall that the boundary of a 2-manifold is a possibly disconnected 1-manifold, hence the countable union of disjoint curves each homeomorphic to either the circle or the real line.
 
 A point $p \in X$ is called a \emph{transit point} if there is a geodesic passing through~$p$ within $X$. The following simple observation shows that such points are abundant in the setting of Theorem~\ref{thm:maingen}.
\begin{lemm}[cf. \cite{Res:93}, p.80]
\label{lemm:trans}
 Let $X$ be as in Theorem~\ref{thm:maingen}. Then transit points are dense in~$X$. More generally, transit points are dense within any simple curve $\gamma\subset X$.
\end{lemm}
Note that, for the latter conclusion to hold, it is important that $X$ be a surface with $\partial X$ composed of piecewise geodesic curves. We remark that transit points do not play any role in the proof of \Cref{thm:main}.  Thus a reader who is not interested in conclusion \ref{item:thm_v} of \Cref{thm:maingen} may ignore all statements about transit points throughout the paper. We include the proof of \Cref{lemm:trans} for convenience of the reader.

\begin{proof}[Proof of Lemma~\ref{lemm:trans}]
Let $\gamma\subset X$ be a simple curve homeomorphic to an interval, $p \in |\gamma|$ and $\varepsilon >0$. We may assume without loss of generality that $p$ is not an endpoint of $\gamma$. Since $\partial X$ is comprised of piecewise geodesic curves, we may furthermore assume that $\gamma \subset X\setminus \partial X$. Now we extend $\gamma$ to a simple closed curve $\widetilde \gamma \subset X \setminus \partial X$ which bounds a disk $U\subset X \setminus \partial X$. Choose $\delta>0$ small enough so that
\[
\delta < \min\left\{ \varepsilon , d\left(p, |\widetilde{\gamma}|\setminus |\gamma|\right)\right\}/2
\]
and such that all pairs of points in $B(p,\delta)$ can be joined by a geodesic within $X$. Now, since $p \notin \partial X$, there exist points $x\in B(p,\delta)\cap U^\circ$ and $y\in B(p,\delta)\setminus U$. Any geodesic from $x$ to $y$ must intersect $\widetilde \gamma$ in a point $z$. By our choice of $\delta$ we must have $z\in B(p,\varepsilon)\cap |\gamma|$. Thus $z$ is the desired transit point.
\end{proof}

 \subsection{Disks and polygons} \label{sec:disks}

 Throughout this section, let $X$ be as in Theorem~\ref{thm:maingen}. That is, $X$ is a length space homeomorphic to a surface such that every component of $\partial X$ is piecewise geodesic. A set $U\subset X$ is a \emph{neighbourhood} of $x\in X$ if~$x$ lies in the topological interior of~$U$ within~$X$, and a \emph{disk} if $U$ is homeomorphic to a closed ball in~$\mathbb{R}^2$. If~$U$ is a disk, then we denote by~$\partial U$ its boundary as a manifold, rather than its topological boundary within $X$, and by $U^\circ$ its interior as a manifold. In particular, if $x \in \partial X$, then for every disk neighborhood $U$ of $x$ we have that $x \in \partial U$. A family $(U_i)_{i\in \mathcal{I}}$ of disks is $\emph{non-overlapping}$ if $U^\circ_i \cap U^\circ_j=\emptyset$ for all distinct $i,j\in \mathcal{I}$. 
 
If $U$ is a disk then, after fixing an orientation on $U$, to every $p\in U$ and every closed curve $c$ in $U\setminus \{p\}$ we can associate an integer called the \emph{winding number}. The winding number is characterized, up to sign, by the following properties:
\begin{itemize}
\item If $c$ is simple and non-contractible in $U \setminus \{p\}$, then the winding number is either~$1$ or~$-1$. \item The winding numbers of two curves agree precisely when they are homotopic within $U\setminus \{p\}$.
\item The winding number is additive with respect to composition of closed curves.
\end{itemize} 
 We further observe that for fixed $p$ the winding number is continuous as a function of~$c$ with respect to uniform convergence. We say that $c$ \emph{winds around $p$} if $c$ is non-contractible in $U\setminus \{p\}$ or, equivalently, if the winding number of $c$ and $p$ is nonzero. We also say that $c$ winds around the set $A\subset U\setminus|c|$ if $c$ winds around every $x\in A$.
 
A polygon is a disk $P \subset X$ with piecewise geodesic boundary $\partial P$, together with a representation of $\partial P$ as a piecewise geodesic curve $e_1*\cdots *e_n$. Each geodesic $e_i$ is called an \textit{edge} of $P$, and each initial point of an edge is called a \textit{vertex}. If $e_{i} * e_{i+1}$ is a geodesic for some $i \in \{1, \ldots, n\}$ (taking $e_{n+1}=e_1$), then $e_i$ and $e_{i+1}$ can be consolidated into a single edge $\widetilde{e}_i$, thus forming a new polygon with $n-1$ edges. A polygon is \textit{reduced} if no such consolidation is possible. For any polygon~$P$, repeating this process of consolidation gives a reduced polygon. A polygon is a \emph{triangle} if it has at most $3$ vertices and a \emph{bigon} if it has exactly $2$ vertices. A triangle is called \emph{degenerate} if it can be reduced to a bigon. Note that whether a triangle~$T$ is degenerate or not depends not only on $T$ as a set but also on the choice of boundary geodesics. For example, let $T$ be the square $[0,1]^2$ equipped with the $\ell^1$-metric. Taking $\{(0,0),(1,1)\}$ as the vertex set gives a representation of $T$ as a bigon, while $\{(1/2,0),(0,1/2),(1,1)\}$ gives a representation of $T$ as a non-degenerate triangle. Indeed, if $B$ is a bigon with vertices $p$ and $q$, and $\partial B$ is locally a geodesic at the vertex $p$, then we can always turn $B$ into a non-degenerate triangle by what we call the \textit{vertex perturbation trick}; see the proof of Lemma III.7 in \cite[p.60-61]{AZ:67}. Namely, replace $p$ by sufficiently close points $p_l$ and $ p_r$ that lie in different connected components of $\partial B \setminus \{p,q\}$. It is easy to check that the three triangle inequalities for the vertices $q,p_l,p_r$ are strict, and hence we obtain a representation of the set $B$ as a non-degenerate triangle.

\section{Boundary convexity}
\label{sec:boundary_convexity}
Throughout this section, let $X$ be as in Theorem~\ref{thm:maingen}. That is, $X$ is a length space homeomorphic to a surface such that every component of $\partial X$ is piecewise geodesic. A set $K\subset X$ is \emph{convex} if for every $x,y\in K$ some geodesic from $x$ to $y$ is contained in $K$, and  \emph{completely convex} if for every $x,y\in K$ every geodesic from $x$ to $y$ is contained in $K$. The following convexity property plays a fundamental role in the proof of \Cref{thm:main}.
\begin{defi} \label{defi:convexity}
A disk $K\subset X$ is \emph{convex relative to its boundary} or \emph{boundary convex} if there is a disk $U$ containing $K$ such that the following hold:
\begin{enumerate}
    \item $d(K,\partial U\setminus \partial X)> 4 \cdot \ell(\partial K)$, \label{item:boundary_convexity_1}
        \item $\diam(U) \leq \diam(X)/3$, and \label{item:boundary_convexity_3}
    \item for every subarc $\gamma$ of $\partial K$ and every curve $\eta$ in $ U\setminus K^\circ$ that is path homotopic to $\gamma$ within $U\setminus K^\circ$, one has $\ell(\gamma)\leq \ell(\eta)$. \label{item:boundary_convexity_2}
\end{enumerate}
\end{defi}
In the situation of \Cref{defi:convexity}, we also say that $K$ is \textit{boundary convex with respect to $U$}. 
It is easy to see that boundary convexity implies convexity. Finally, the disk $K$ is \emph{absolutely convex} if it is boundary convex and completely convex.  

Note that boundary convexity is called ``bounded convexity'' in \cite{Res:93}.  Our definition is slightly more restrictive than the ones given on p.~48 of \cite{AZ:67} and on p.~80 of \cite{Res:93}. The main difference is that we have added condition (\ref{item:boundary_convexity_3}). This ensures that if $U_1$ and $U_2$ are both ambient disks satisfying conditions~(\ref{item:boundary_convexity_1}) and ~(\ref{item:boundary_convexity_3}) for a given disk~$K$, then condition~(\ref{item:boundary_convexity_2}) holds for $U_1$ if and only if it holds for $U_2$. Condition~(\ref{item:boundary_convexity_1}) alone does not suffice to guarantee this independence, because without (\ref{item:boundary_convexity_3}) a curve $\eta\subset U_1\cap U_2$ may homotope to different subarcs of $\partial K$ in the respective ambient disks $U_1$ and $U_2$. This unpleasant behavior can also be avoided by assuming that $X$ is not homeomorphic to $\mathbb{S}^2$. 


The proof of \Cref{thm:main} depends in an essential way on boundary convexity, as opposed to convexity or complete convexity. The reason is that boundary convexity is preserved by certain operations of both intersection and subdivision. This is the content of the next two lemmas.
\begin{lemm}
\label{lemm:intersect}
Let $K_1$ and $K_2$ be boundary convex disks. If $W$ is the closure of a connected component of the interior of $K_1\cap K_2$, then $W$ is a boundary convex disk. 
\end{lemm} 
Note that \Cref{lemm:intersect} fails when boundary convexity is replaced by mere convexity. To obtain a counterexample, we start with a metric graph as in \Cref{fig:convexity} with the property that $\ell(\beta_1) = \ell(\beta_2) < \ell(\alpha_1) = \ell(\alpha_2)$. We construct a surface~$X$ by gluing in  three round hemispheres isometrically along their boundary curves and, respectively, $\bar{\beta}_1*\alpha_1$, $\bar{\alpha}_1*\alpha_2$ and $\bar{\alpha}_2*\beta_2$. Let $K_1$ be the polygon bounded by $\beta_1$ and~$\alpha_2$,  and let $K_2$ be the polygon bounded by $\alpha_1$ and $\beta_2$. Then $K_1$ and~$K_2$ are both convex, while the intersection $K_1 \cap K_2$ is the hemisphere bounded by $\alpha_1$ and~$\alpha_2$, which is not convex. The intersection property in \Cref{lemm:intersect} does not appear explicitly in \cite{AZ:67}, but we need it to work around Lemma~III.6 in \cite{AZ:67}. See \Cref{subsec:covover} below.

\begin{figure} 
    \centering
    \begin{tikzpicture}[scale=6]
    \draw[orange,thick] (0,0) .. controls (.2,.4) and (.8,.4) .. (1,0);
    \draw[darkgray,thick] (0,0) .. controls (.4,.15) and (.6,.15) .. (1,0);
    \draw[red,thick] (0,0) .. controls (.2,.-.4) and (.8,-.4) .. (1,0); 
    \draw[blue,thick] (0,0) .. controls (.4,.-.15) and (.6,-.15) .. (1,0);
    \filldraw (0,0) circle (.225pt);
    \filldraw (1,0) circle (.225pt);
    \filldraw[white] (.5,.17) circle (.2pt);
    \node[red] at (.165,-.25) {\Large $\beta_2$};
    \node[orange] at (.165,.25) {\Large $\beta_1$};
    \node[blue] at (.3,-.125) {\Large $\alpha_2$};
    \node[darkgray] at (.3,.125) {\Large $\alpha_1$};
\end{tikzpicture}
    \caption{}
    \label{fig:convexity}
\end{figure}

For the proof, we recall that the subspace metric on a convex subset of a length metric space itself defines a length metric on the subset.
\begin{proof}
Let $W$ be the closure of a connected component of the interior of $K_1 \cap K_2$, and let $U_1, U_2$ denote the ambient disk neighborhoods of $K_1, K_2$, respectively, as in \Cref{defi:convexity}. We assume without loss of generality that $\ell(\partial K_1) \geq \ell(\partial K_2)$. Then $\partial K_2\subset U_1 \cap U_2$. Note that, by condition (\ref{item:boundary_convexity_3}), whenever $c\subset U_1\cap U_2$ is a Jordan curve, then the disk bounded by~$c$ within~$U_1$ is the same as the disk bounded by~$c$ within~$U_2$. Compare also the discussion after \Cref{defi:convexity}. Applying this observation to $\partial K_2$, we conclude that $K_2\subset U_1$.

To prove that $W$ is a disk we will employ Kerékjártó's theorem (see e.g. \cite[p.168]{New:64}). Kerékjártó's theorem states that whenever $\gamma_1$ and $\gamma_2$ are simple closed curves in $\mathbb{R}^2$ that have more than one point in common, then the closure of every bounded complementary component of $|\gamma_1|\cup |\gamma_2|$ is a disk. Note that $W$ is also the closure of some complementary component of $\partial K_1\cup \partial K_2$ in $U_1$. If $W$ is not equal to $K_1$ or $K_2$,  then $\partial K_1$ and $\partial K_2$ have more than one point in common. Hence Kerékjártó's theorem implies that $W$ is a disk.

It remains to show that $W$ is boundary convex with respect to $U_1$. Certainly condition (\ref{item:boundary_convexity_3}) is satisfied. Assume that condition (\ref{item:boundary_convexity_2}) fails. Then there are a subarc $\widetilde{\gamma}$ of $\partial W$ and a curve $\widetilde{\eta} \subset U_1\setminus W^\circ$ that are path homotopic within $U_1\setminus W^\circ$ and such that $\ell(\widetilde{\eta})<\ell(\widetilde{\gamma})$. By the Arzel\`a--Ascoli theorem, lower semicontinuity of length and continuity of the winding number, we may assume that $\widetilde{\eta}$ is shortest among all curves that are path homotopic to $\widetilde{\gamma}$ within $U_1 \setminus W^\circ$. 
By the boundary convexity of $K_1$ with respect to $U_1$, we may assume that $|\widetilde{\eta}|\subset K_1$. There must be subcurve $\eta$ of $\widetilde{\eta}$ such that $\eta$ intersects $\partial W$ only in its endpoints and $\ell(\eta)<\ell(\gamma)$, where $\gamma$ is the subarc of $\partial W$ that is path homotopic to $\eta$ within $U_1\setminus W^\circ$. Otherwise, $\widetilde{\eta}$ would be path homotopic within $U_1\setminus W^\circ$ to a curve $\widetilde{\nu}$ that is contained in $\partial W$ and satisfies $\ell(\widetilde{\nu})\leq \ell(\widetilde{\eta})$. This would be a contradiction since then $|\widetilde{\gamma}|\subset |\widetilde{\nu}|$ and hence 
\[
\ell(\widetilde{\gamma})=\mathcal{H}^1(|\widetilde{\gamma}|)\leq \mathcal{H}^1(|\widetilde{\nu}|)\leq \ell(\widetilde{\nu})\leq \ell(\widetilde{\eta}).
\]
Note that $\gamma$ must be contained in $\partial K_2$. Thus, if $\eta$ intersected the entire disk $K_2$ only at its endpoints, then we could apply the boundary convexity of $K_2$ to derive that $\gamma$ is a geodesic and hence obtain a contradiction. However, since this might \textit{a priori} not be true, we must work harder. See \Cref{fig:boundary_convex} for an illustration of such a potential situation.

\begin{figure}
\begin{tikzpicture}[scale=2.8]
    \fill[gray,opacity=.2] (.5,.2) to (1.5,.2) to (1.5,.6) to (.5,.6);
    \draw[thick,blue] (.5,1.4) to (1.5,1.4) to (1.5,0) to (.5,0) to (.5,1.4);
    \draw[thick,red] (.5,1.2) to (2.1,1.2) to (2.1,.2) to (.5,.2) .. controls (.2,.2) and (.2,.6) .. (.5,.6) to (1.5,.6) .. controls (1.7,.6) and (1.7,.8) .. (1.5,.8) to (.5,.8) .. controls (.2,.8) and (.2,1.2) .. (.5,1.2);
    \draw[thick] (.7,.6) to (1.3,.6);
    \draw[thick,orange] (.7,.6) .. controls (.7,1.1) and (1.3,1.1) .. (1.3,.6);
    \filldraw[blue] (.5,1.4) circle (.5pt);
    \filldraw[blue] (1.5,1.4) circle (.5pt);
    \filldraw[blue] (1.5,0) circle (.5pt);
    \filldraw[blue] (.5,0) circle (.5pt);
    \filldraw[red] (.6,.2) circle (.5pt);
    \filldraw[red] (.6,.6) circle (.5pt);
    \filldraw[red] (.6,.8) circle (.5pt);
    \filldraw[red] (.6,1.2) circle (.5pt);
    \filldraw[red] (2.1,.2) circle (.5pt);
    \filldraw[red] (1.4,.6) circle (.5pt);
    \filldraw[red] (1.4,.8) circle (.5pt);
    \filldraw[red] (2.1,1.2) circle (.5pt);
    \filldraw[] (.7,.6) circle (.5pt);
    \filldraw[] (1.3,.6) circle (.5pt);
    \node[blue] at (.375,1.4) {$K_1$};
    \node[red] at (1.9,.1) {$K_2$};
    \node[] at (1.1,.525) {$\gamma$};
    \node[orange] at (1.15,1.) {$\eta$};
    \node at (.8,.3) {$W$};
    \end{tikzpicture}
    \caption{}
    \label{fig:boundary_convex}
\end{figure}

First, we claim that $\eta$ is simple. To verify this, note that both endpoints of $\eta$ lie on a simple arc $A\subset K_1 \cap \partial K_2\cap \partial W$  that contains $|\gamma|$ and separates $|\eta| \setminus A$ from $W^\circ$ within $K_1$. 
Thus, if $\eta$ had a self-intersection, then we could shorten $\eta$ by deleting some subcurve of it. Note that $A$ would still separate the resulting curve from $W^\circ$  within $K_1$ and thus this curve must also be path homotopic to $\gamma$ within $U_1\setminus W^\circ$. This would contradict the length minimality of $\eta$. 
Let $c=\eta*\bar{\gamma}$ and denote by $O$ the complementary component of $|\eta|\cup \partial K_1 \cup \partial K_2$ that is adjacent to $\gamma$ and differs from $W$. Note that $|c| \subset U_1\setminus O$ and $c$ winds around $O$ within $U_1$. Denote by $(\eta_i)_{i\in \mathcal{I}}$ the closures of the connected components of $|\eta| \setminus K_2$, and for each $i$ denote by $\gamma_i$ the respective subarc of $\partial K_2$ which is path homotopic to $\eta_i$ within $U_1\setminus  K_2^\circ$. Note that, again by condition~\eqref{item:boundary_convexity_3}, $\gamma_i$ is also path homotopic to $\eta_i$ within $U_2\setminus K_2^\circ$, and hence we conclude by the boundary convexity of $K_2$ that $\ell(\gamma_i)\leq \ell(\eta_i)$. This in turn implies that each $\gamma_i$ is path homotopic to $\eta_i$ within $U_1\setminus O$, for otherwise $|\gamma|\subset |\gamma_i|$ and hence \[\ell(\gamma)\leq \ell(\gamma_i)\leq \ell(\eta_i)\leq \ell(\eta).\] Thus the curve $\widehat{c}$ obtained by replacing in $c$ each $\eta_i$ with $\gamma_i$ is path homotopic to $c$ within $U_1\setminus O$. However, $|\widehat{c}|$ is contained in $K_2$, and hence we conclude that $c$ is contractible in $U_1\setminus O$. This gives a contradiction, since we had initially observed that $c$ winds around $O$ within $U_1$.

Finally, we verify condition (\ref{item:boundary_convexity_1}). It suffices to show that $\ell(\partial W)\leq \ell(\partial K_2)$, since then
\[
d(W,\partial U_1 \setminus \partial X)\geq d(K_1,\partial U_1 \setminus \partial X)>4 \cdot \ell(\partial K_1)\geq 4\cdot \ell(\partial K_2)\geq \ell(\partial W).
\]
To this end, let $(\alpha_i)_{i\in \mathcal{I}}$ be the countable family of closures of the connected components of $\partial W\setminus \partial K_2$. Then each $\alpha_i$ is a subarc of $\partial K_1$ and contained in $K_2$ and intersects $\partial K_2$ precisely at its endpoints. Let  $\beta_i$ be the subarc of $\partial K_2$ which is path homotopic to $\alpha_i$ within $U_1\setminus W^\circ$. By condition (\ref{item:boundary_convexity_2}),  we have $\ell(\alpha_i)\leq \ell(\beta_i)$. Note that distinct $\beta_i$ and $\beta_j$ can intersect at most in their endpoints and that $\beta_i$ intersects $\partial W$ only in its endpoints. We conclude that $\ell(\partial W)\leq \ell(\partial K_2)$.
\end{proof}

\begin{lemm}
\label{lemm:part}
Let $P\subset X$ be a boundary convex polygon and $\gamma$ a geodesic in $X$ with endpoints in $X\setminus P^\circ$. If $Q$ is the closure of some connected component of $P\setminus |\gamma|$, then $Q$ is a boundary convex polygon. Furthermore, $\partial Q\setminus \partial P$ consists of transit points.
\end{lemm}
See \cite[Lem. III.2, p.49]{AZ:67} for a slightly weaker result. Note that \Cref{lemm:part} holds when boundary convexity is replaced by convexity, but not when replaced by complete convexity.
\begin{proof}
Let $U$ be an ambient disk such that $P$ is boundary convex with respect to~$U$. Since otherwise the claim is obvious, we may assume that $Q\neq P$ and, by possibly deleting initial and terminal subcurves, that $\gamma$ has endpoints in~$\partial P$. Notice that $Q$ still appears as the closure of some complementary component when replacing  all portions of $\gamma$ that lie outside $P$ by the respective homotopic subcurves of $\partial P$. Hence, by the boundary convexity of $P$, we may assume that $|\gamma| \subset P$. 
Then $Q$ is a polygon with boundary comprised of a subarc  $\eta$ of $\partial P$ and a subcurve $\alpha$ of $\gamma$. 

We show that $Q$ is boundary convex with respect to~$U$. Clearly Condition~\eqref{item:boundary_convexity_3} is satisfied. That $\alpha$ is a geodesic implies Condition~\eqref{item:boundary_convexity_1}, since
\[
d(Q,\partial U\setminus \partial X)\geq d(P,\partial U\setminus \partial X)> 4\cdot \ell(\partial P)\geq 4\cdot(\ell(\alpha)+\ell(\eta))=4\cdot \ell(\partial Q).
\]
To check condition (\ref{item:boundary_convexity_2}), let $c$  be a simple subcurve of $\partial Q$ and $\nu$ be a simple curve in $U \setminus Q^\circ$ that is path homotopic to $c$ within $U\setminus Q^\circ$. 
By the boundary convexity of $P$, we may find a curve $\nu_1\subset P\setminus Q^\circ$ that is path homotopic to $\nu$ within $U\setminus Q^\circ$ and satisfies $\ell(\nu_1)\leq \ell(\nu)$. 
The closure of each connected component of $\nu_1\setminus Q$ is moreover path homotopic within $U\setminus Q^\circ$ to some subarc of the geodesic~$\alpha$. Thus, we may find a curve $\nu_2$ in $\partial Q$ that is path homotopic to $\nu_1$ within $U\setminus Q^\circ$ and satisfies $\ell(\nu_2)\leq \ell(\nu_1)$. Since $c$ and $\nu_2$ are homotopic within $U\setminus Q^\circ$, both contained in $\partial Q$ and $c$ is simple, we must have that $|c|\subset |\nu_2|$. Then
\[
\ell(c) = \mathcal{H}^1(|c|)\leq \mathcal{H}^1(|\nu_2|) \leq \ell(\nu_2)\leq \ell(\nu_1)\leq \ell(\nu).
\]
We conclude that $Q$ is boundary convex with respect to $U$.
\end{proof}

\section{Superfluous intersections of geodesics}
 \label{sec:sup}
Let $X$ be as in the statement of Theorem~\ref{thm:maingen}. That is, $X$ is a length space homeomorphic to a surface such that every component of $\partial X$ is piecewise geodesic. A family~$(\gamma_i)_{i\in \mathcal{I}}$ of geodesics in $X$ \textit{does not have superfluous intersections} if for every $i,j \in \mathcal{I}$ the intersection $|\gamma_i|\cap |\gamma_j|$ is connected. Similarly, we say that an additional geodesic $\gamma$ \textit{does not have superfluous intersections with} $(\gamma_i)_{i \in \mathcal{I}}$ if for every $i \in \mathcal{I}$ the intersection $|\gamma| \cap |\gamma_i|$ is connected. It is claimed on p. 51 of~\cite{AZ:67} and p. 79 of~\cite{Res:93} that, given points $x,y\in X$ and a  finite system of geodesics $(\gamma_i)_{i=1}^k$ that does not have superfluous intersections, then one can always find a geodesic joining $x$ to $y$ that does not have superfluous intersections with $(\gamma_i)_{i=1}^k$. This claimed observation is frequently used in \cite{AZ:67}. However, it turns out to be false in general. As a counterexample, consider a surface $X$ containing geodesics $\gamma_1, \ldots, \gamma_4$ that intersect as pictured in \Cref{fig:superfluous_intersection}, with the property that any geodesic connecting the pictured points $x$ and $y$ is contained in $|\gamma_1| \cup \cdots \cup |\gamma_4|$. Then one can check that $\Gamma= (\gamma)_{i=1}^4$ does not have superfluous intersections, but any geodesic from $x$ to $y$ must have superfluous intersections with $\Gamma$.

\begin{figure} 
    \centering
    \begin{tikzpicture}[scale=8]
    \draw[orange,thick] (0,0) .. controls (0,.1) and (.25,.11) .. (.5,.05) to (.8,-.03);
    \draw[darkgray,thick] (.5,.05) .. controls (.55,.125) and (.8,.12) .. (.9,.05) to (1,0);
    \draw[red,thick] (0,0) to (1,-.1) .. controls (1.2,.-.1) and (1.2,.1) .. (1,.1) .. controls (.95,.1) and (.9,.08) .. (.9,.05); 
    \draw[blue,thick] (.6,-.06) to (1,0);
    \filldraw (0,0) circle (.2pt);
    \filldraw (1,0) circle (.2pt);
    \filldraw[white] (.5,.17) circle (.2pt);
    \node at (-.03,-.02) {\Large $x$};
    \node at (1.03,-.02) {\Large $y$};
    \node[red] at (.25,-.06) {\Large $\gamma_1$};
    \node[orange] at (.12,.11) {\Large $\gamma_2$};
    \node[blue] at (.89,-.05) {\Large $\gamma_3$};
    \node[darkgray] at (.7,.07) {\Large $\gamma_4$};
\end{tikzpicture}
    \caption{}
    \label{fig:superfluous_intersection}
\end{figure}

To overcome this complication, we prove two weaker results. The first is the following.

\begin{lemm}
\label{lemm:sup2}
Let $P\subset X$ be a polygon with edges $(e_1,\dots ,e_n)$, where $n\geq 3$, and let $p_0,p_1,p_2 \in P$. Assume further that $P$ is contained in a disk $U$ with $d(P,\partial U \setminus \partial X)>\diam(P)$. Then there is a geodesic $\gamma_1$ from $p_0$ to $p_1$ that does not have superfluous intersections with~$(e_1,\dots,e_n)$. In addition, for any such $\gamma_1$, there is a geodesic~$\gamma_2$ from $p_0$ to $p_2$ that does not have superfluous intersections with $(e_1, \dots , e_n,\gamma_1)$.
\end{lemm}
Note that the respective result for three geodesics emanating from~$p_0$ fails by a counterexample similar to the construction in~\Cref{fig:superfluous_intersection}.
\begin{proof}
Since $U$ is compact and $d(P,\partial U \setminus \partial X)>\diam(P)$, the points $p_0$ and $p_1$ are joined by some geodesic $\eta_0$. Now we inductively construct geodesics $\eta_i$ from $p_0$ to $p_1$ so that $\eta_i$ does not have superfluous intersections with $(e_1, \dots , e_i)$. The first part of the claim then follows by setting $\gamma_1= \eta_n$. Assume that such $\eta_{i}$ has been constructed for some $0\leq i\leq n-1$. If $\eta_{i}$ does not intersect $e_{i+1}$, we set $\eta_{i+1}=\eta_{i}$. Otherwise, let $l$ be the first point of intersection of $\eta_{i}$ and $e_{i+1}$, and $r$ be the last one. Now $\eta_{i+1}$ is obtained from $\eta_{i}$ by replacing the portion of $\eta_{i}$ between $l$ and $r$ by the respective one of $e_{i+1}$. Clearly $|\eta_{i+1}| \cap |e_{i+1}|$ is connected. Furthermore, since $n\geq 3$, we have for $j\leq i$ that $|e_j|\cap |\eta_{i+1}|$ is either empty or equal to $|e_j| \cap |\eta_{i}|$, and hence connected.

Now fix a geodesic $\gamma_1$ of the described type. The previous argument also provides us with a geodesic $\widetilde{\eta}$ from $p_0$ to $p_2$ not having superfluous intersections with $(e_1,\dots,e_n)$. Let $q$ be the last point of intersection of $\widetilde{\eta}$ and $\gamma_1$. This setup is shown in \Cref{fig:polygon_superfluous}. If $q=p_2$, we may choose $\gamma_2$ as a subcurve of $\gamma_1$, and, if $q=p_0$, we may set $\gamma_2=\widetilde{\eta}$. Otherwise, let $\eta$ be the geodesic which is obtained from $\widetilde{\eta}$ upon replacing the portion of $\widetilde{\eta}$ between $p_0$ and $q$ with the respective portion of $\gamma_1$. By reparametrizing, we may assume that $\eta$ is parametrized on the interval $[0,2]$ so that $\eta(0)=p_0$, $\eta(1)=q$ and $\eta(2)=p_2$. If $\eta$ does not have superfluous intersections with $(e_1,\dots,e_n,\gamma_1)$ then we may set $\gamma_2=\eta$. Otherwise, since $\eta$ does not have superfluous intersections with $\gamma_1$ and the restrictions of $\eta$ to $[0,1]$ and to $[1,2]$ respectively do not have superfluous intersections with $(e_1,\dots, e_n)$, there must exist $s<1$ and $t>1$ such that $\eta(s)$ and $\eta(t)$ lie on a common edge $e_i$ with $q\notin |e_i|$. We may choose $s$ as the minimal one for which such $t$ and edge $e_i$ exists and choose $t$ maximal for the given value $s$. The curve $\gamma_2$ is obtained from $\eta$ by replacing $\eta|_{[ s,t]}$ with the respective portion of $e_i$. It remains to show that $\gamma_2$ does not have superfluous intersections with $(e_1, \dots, e_n, \gamma_1)$.

\begin{figure}
\begin{tikzpicture}[scale=2.8]
    \draw[thick,blue] (.5,1.4) to (1.5,1.1) to (1.5,0) to (.5,0) to (.5,1.4);
    \draw[thick,red] (.7,0) .. controls (.5,.2) .. (.5,.3) to (.5,.35) .. controls (.5,.5) and (1.5,.6) .. (1.5,.75) to (1.5,.8) .. controls (1.5,.9) .. (1.167,1.2);
    \draw[thick] (.7,0) .. controls (.8,.6) and (.8,1) .. (.8,1.1) .. controls (.8,1.6) .. (.5,1.6) .. controls (0.2,1.6) and (0.2,1.2) .. (.5,1.15);
    \filldraw[blue] (.5,1.4) circle (.5pt);
    \filldraw[blue] (1.5,1.1) circle (.5pt);
    \filldraw[blue] (1.5,0) circle (.5pt);
    \filldraw[blue] (.5,0) circle (.5pt);
    \filldraw[] (.7,0) circle (.5pt);
    \filldraw[] (.5,1.15) circle (.5pt);
    \filldraw[] (1.167,1.2) circle (.5pt);
    \filldraw[] (.7675,.4875) circle (.5pt);
    \node[blue] at (1.2,.15) {$P$};
    \node[] at (.725,.8) {$\widetilde{\eta}$};
    \node[red] at (1.15,.5) {$\gamma_1$};
    \node at (.81,.4) {$q$};
    \node at (.7,-.1) {$p_0$};
    \node at (1.2,1.275) {$p_1$};
    \node at (.6,1.15) {$p_2$};
    \end{tikzpicture}
    \caption{}
    \label{fig:polygon_superfluous}
\end{figure}

First of all, since $|\gamma_1\cap \eta|=\eta([0,1])$, we have 
\begin{equation}
\label{eq:setin}
\gamma_2([0,s])\subset |\gamma_2| \cap |\gamma_1| \subset \gamma_2([0,t]).
\end{equation}
However, $\gamma_2([s,t])$ is a subgeodesic of $e_i$ and hence, since $\gamma_1$ does not have superfluous intersections with $e_i$, it follows that $\gamma_2([ s,t])\cap |\gamma_1|$ is connected and hence, by \eqref{eq:setin}, that so is $|\gamma_2|\cap |\gamma_1|$. Next, we have by the minimal and maximal choice of $s$ and $t$ that $|\gamma_2| \cap |e_i|=\gamma_2([s,t])$ and hence that $|\gamma_2| \cap |e_i|$ is connected. 
Finally, let $j\neq i$. Since the edges of $P$ only intersect in their endpoints we must have $|e_j|\cap |\gamma_2| \subset |\gamma_2|\setminus \gamma_2((s,t))$. However, by the maximal and minimal choice of $s$ and $t$, we must indeed have that $|e_j|\cap |\gamma_2|$ is either contained in $\gamma_2([0,s])$ or in $\gamma_2([t,1])$. Since the restrictions of $\gamma_2$ to these subintervals do not have superfluous intersections with $e_j$, it follows also that $|e_j|\cap |\gamma_2|$ is connected.
\end{proof}

\begin{rem}
\label{rem:sup_big}
The conclusion of \Cref{lemm:sup2} is false whenever $n=2$ and $p_0,p_1$ are the two vertices of $P$. Nevertheless, the proof also shows that the result remains true for bigons~$P$ if we additionally require that none of the points $p_0,p_1,p_2$ is a vertex point of~$P$.
\end{rem}
In many situations, it is not necessary to achieve that $|\gamma_i|\cap |\gamma_j|$ is connected, but it suffices that the intersection $|\gamma_i|\cap |\gamma_j|$ has finitely many connected components. We say that a system $\Gamma=(\gamma_i)_{i=1}^k$ of geodesics is a \emph{finite graph} if $|\Gamma|:=|\gamma_1|\cup \dots \cup |\gamma_k|$ is a finite topological graph when endowed with the subspace topology. Equivalently, $(\gamma_i)_{i=1}^k$ is a finite graph if $|\gamma_i| \cap |\gamma_j|$ has only finitely many connected components for every $1\leq i,j \leq k$.
\begin{lemm}
\label{lemm:sup1}
Let $\Gamma=(\eta_i)_{i=1}^k$ be a finite graph and $\gamma^*$ an additional geodesic in $X$. Then there is a geodesic~$\gamma$ satisfying the following: 
\begin{enumerate}[label=(\roman*)]
\item $\gamma$ has the same endpoints as $\gamma^*$,
\item $\gamma \cup \Gamma$ is a finite graph, and
\item \label{i3} every connected component of $|\gamma| \setminus |\gamma^*|$ is contained in the image of some geodesic $\eta_i \in \Gamma$.
\end{enumerate}
\end{lemm}
\begin{proof}
We prove the claim by induction on $k$. For the base case $k=0$, we simply set $\gamma=\gamma^*$. So assume the claim of the lemma holds for all  finite graphs comprising at most $k-1$ geodesics. Now let $\Gamma=(\eta_i)_{i=1}^k$ be a finite graph and set $\widehat{\Gamma}=(\eta_i)_{i=1}^{k-1}$. By the induction assumption, we can find a geodesic $\widehat{\gamma}$ with the same endpoints as~$\gamma^*$ such that $\widehat{\gamma}\cup \widehat{\Gamma}$ is a finite graph and each connected component of $|\widehat{\gamma}| \setminus |\gamma^*|$ is contained in the image a single geodesic $\eta_i\in \widehat{\Gamma}$.

If $|\widehat{\gamma}|\cap |\eta_k|$ has only finitely many connected components, then we may simply set $\gamma=\widehat{\gamma}$. So assume otherwise. Then, since $\eta_k \cup \widehat{\Gamma}$ and $\widehat{\gamma} \cup \widehat{\Gamma}$ are finite graphs, all except finitely many of these connected components are contained in $|\widehat{\gamma}| \setminus |\widehat{\Gamma}|\subset |\gamma^*|$. Let $l$ be the first point in $|\widehat{\gamma}|\cap |\gamma^*| \cap |\eta_k|$ and $r$ be the last one. Let $\gamma$ be the geodesic obtained from $\widehat{\gamma}$ upon replacing the portion between $l$ and $r$ by the respective portion of $\eta_k$. Then $\gamma$ has the same endpoints as $\widehat{\gamma}$, and hence also as $\gamma^*$. The intersection of $\gamma$ with each $\eta_i\in \Gamma$ has only finitely many connected components, since $\gamma$ is the composition of three curves, each having this property. In particular $\gamma \cup \Gamma$ is a finite graph. Finally, each connected component of $|\gamma| \setminus |\gamma^*|$ is either contained in $|\eta_k|$, or is equal to a connected component of $|\widehat{\gamma}| \setminus |\gamma^*|$ and thus contained in the image of a single geodesic $\eta_i\in \widehat{\Gamma}\subset \Gamma$. 
\end{proof}

\section{Proof of the main theorem}
\label{sec:proof}

We now proceed with the proof of \Cref{thm:maingen}, with the exception of the nondegeneracy conclusion \ref{item:non_degenerate}, which is postponed until \Cref{sec:degenerate_triangles}. Throughout this section, let $X$ be as in \Cref{thm:maingen}. That is, $X$ is a length surface with $\partial X$ composed of piecewise geodesic curves.
 \subsection{Covering by polygons of small diameter}
 \label{subsec:diamcov}
 
 We begin by showing that $X$ may be covered by polygons having arbitrarily small diameter.
\begin{lemm}[cf. \cite{AZ:67}, Lem. III.3, p.51]
\label{lemm1}
Let $x\in X$ and $\varepsilon >0$. Then there is a polygonal neighbourhood $P$ of~$x$ such that $\diam(P)\leq \varepsilon$ and $\partial P\setminus \{x\}$ consists of transit points.
\end{lemm}
The idea of the proof is the following. We take a small disk neighbourhood $V$ of $x$ and then choose a sufficiently fine finite set of points in $\partial V$. Connecting each consecutive pair of these points by a geodesic gives a piecewise geodesic curve $\gamma$. If $\gamma$ is a Jordan curve, then $P$ is the disk bounded by $\gamma$. In general, however, we must carefully delete some portions of $\gamma$ to turn it into a Jordan curve. 
\begin{proof}
Let $U$ be a disk neighbourhood of $x$ such that $\diam(U)\leq \varepsilon$ and $U \cap \partial X$ contains no vertices of $\partial X$ except for possibly $x$ itself. Let $V\subset U$ be a disk neighbourhood of $x$ such that $d(V ,\partial U \setminus \partial X)>0$. 
 
Choose some $\delta$ satisfying $0<\delta< d(V,\partial U\setminus \partial X)$, to be determined later. Now take a finite collection of distinct points $\{y_j\}_{j=0}^m \subset \partial V$, labelled in cyclic order, such that $m\geq 2$ and $\diam (|\eta_j|)< \delta$ for all~$j$, where~$\eta_j$ denotes the arc of~$\partial V$ between~$y_j$ and~$y_{j+1}$. Furthermore, if $x\in \partial X$, we assume that $y_0=x$. Applying \Cref{lemm:trans}, by perturbing the points if needed, we can choose each $y_j$, except for possibly $x$, to be a transit point. Connect each $y_j$ to $y_{j+1}$ by a geodesic~$e_j$. Such~$e_j$ exists and is contained in~$U$ since $d(y_j,y_{j+1})<d(V,\partial U\setminus \partial X)$. Whenever $|\eta_j|\subset \partial X$, then $\eta_j $ is a geodesic and we choose $e_j=\eta_j$.
  Let $\gamma=e_0*\cdots*e_m$. Then $\gamma$ defines a closed piecewise geodesic curve that is contained in~$U$ and such that $| \gamma|\setminus \{x\}$ consists of transit points. By Lemma~\ref{lemm:sup1}, we may furthermore assume that $|\gamma|$ is a finite topological graph.
 
We now consider two cases. First, suppose that $x\in X\setminus \partial X$. In this case, set $\mu=1/2\cdot d(x,\partial V)>0$.  As noted in the proof of \cite[Lemma 9.4]{PS:19}, we may choose $0<\delta<\mu$ such that a compact subset of~$U$ of diameter at most~$\delta$ cannot separate a subset of~$U$ of diameter at least~$\mu$ from the boundary curve $\partial U$.  
To prove that such $\delta$ exists, we identify $U$ with the unit disk in $\mathbb{R}^2$ and denote by $\bar{\partial} E$ the boundary of the set $E\subset U$ as a subset of $\mathbb{R}^2$. Assume such $\delta$ does not exist; then we can find sequences $(E_i)_{i=1}^\infty$ and $(V_i)_{i=1}^\infty$ of compact subsets of $U$ such that $\diam E_i \to 0$, $\diam V_i \geq \mu$ and $\bar{\partial} V_i\subset E_i$. By the Blaschke selection theorem (see e.~g.\ \cite[Theorem 7.3.8]{BBI:01}), by passing to a subsequence we may assume without loss of generality that $(V_i)$ and $(E_i)$ respectively converge in the Hausdorff distance to compact subsets $V$ and $E$ of $U$. Since the diameter is continuous with respect to Hausdorff convergence, we must have that $E=\{p\}$ consists of a single point and that $\diam V\geq \mu$. On the other hand, the boundary is lower semicontinuous with respect to Hausdorff convergence and hence 
\[
\bar{\partial} V\subset \lim_{i\to \infty} \bar{\partial} V_i\subset \lim_{i\to \infty} E_i =\{p\}.
\]
This gives a contradiction, since the boundary of a compact planar set of positive diameter must have cardinality at least two.

For each $j \in \{0,\dots , m\}$, set $c_j=\overline{e}_j* \eta_j$. Then $\diam(|c_j|)\leq \delta$ and hence, since \[
d(x,|c_j|)\geq d(x,\partial V)-\diam (|\eta_j|)\geq 2\mu-\delta >\mu,
\] the curve $|c_j|$ cannot separate $x$ from $\partial U$. Thus $\eta_j$ is path homotopic to $e_j$ within $U\setminus \{x\}$, and hence $\partial V$ is path homotopic to $\gamma$ within $U\setminus \{x\}$. In particular, $|\gamma|$ must separate~$x$ from~$\partial U$. 
Since $|\gamma|$ is a finite topological graph, there is a Jordan curve $\widehat{\gamma}$ with $|\widehat{\gamma}| \subset |\gamma|$ also separating~$x$ from~$\partial U$; see e.g. \cite[Theorem~IV.6.7]{Wil:49}. The curve $\widehat{\gamma}$ must be piecewise geodesic curve as well. The claim follows by taking $P\subset U$ to be the disk bounded by $\widehat{\gamma}$.

Next, suppose that $x\in \partial X$. Denote by $\nu$ the connected component  of $\partial V\cap \partial X$ that contains $x$. In this case, let $\delta< d(x,\partial V\setminus \nu)$. Then each curve $e_j$ is either equal to $\eta_j$, or cannot pass through $x$. Hence the curve $\gamma$ passes through $x$ exactly once. Furthermore, note that $e_0=\eta_0$ and $e_m=\eta_m$. Again, by \cite[Theorem IV.6.7]{Wil:49} we may find a piecewise geodesic Jordan curve $\widehat{\gamma}$ such that $|\widehat{\gamma}|\subset |\gamma|$ and $|\widehat{\gamma}|$ contains a neighbourhood of $x$ within $\partial X$. The claim follows by taking $P\subset U$ to be the polygon corresponding to~$\widehat{\gamma}$.
\end{proof}
 \subsection{Covering by polygons of small perimeter}
 \label{subsec:covper}
The next step is to cover $X$ by polygons that have not only small diameter but also small perimeter. In the original proof of Alexandrov--Zalgaller for surfaces of bounded curvature, this is deduced from the fact that every point in such a surface has a polygonal neighbourhood of small perimeter; see \cite[Lem.~III.5, p.53]{AZ:67} or \cite[Lem.~6.3.3]{Res:93}. However, as noted in \cite{AZ:67} and \cite{Res:93}, this fact does not generalize to arbitrary metric surfaces. One possible counterexample is to take $X$ to be the quotient metric space obtained from the cylinder $\mathbb{S}^1 \times [0,1]$ by collapsing one of its boundary circles to a single point. Instead, we show the following lemma, which is the main novel ingredient of our proof.
\begin{lemm}
\label{lemm3}
Let $x\in X$ and $\varepsilon >0$. Then there is a neighbourhood $U$ of $x$ such that $\diam(U)\leq \varepsilon$ and $U=T_1\cup \cdots \cup T_n$ where each $T_i$ is a triangle such that $\partial T_i\setminus \partial X$ consists of transit points.
\end{lemm}
Note that each triangle $T_i$ has perimeter at most $3\varepsilon$ and that we do not require the triangles to be non-overlapping.

\begin{proof}
Let $V$ be a disk neighbourhood of $x$ such that $\diam(V)\leq \varepsilon$. By Lemma~\ref{lemm1}, there exists a polygonal neighbourhood $P$ of $x$ such that $\diam (P)< d(P,\partial V\setminus \partial X)$ and $\partial P \setminus \{x\}$ consists of transit points. Note that $x\in \partial P$ only when $x\in \partial X$. Let $\partial P= e_0*\cdots *e_{n+1}$ be a representation of $\partial P$ as a piecewise geodesic curve, and let $v_0, \dots, v_{n+1}$ be the corresponding vertices of~$P$, where $v_i$ is the initial point of~$e_i$. By iterated application of \Cref{lemm:sup2}, we may choose for each $i=1,\dots, n+1$ a geodesic $\gamma_i$ from $v_0$ to~$v_i$ such that $(e_0,\dots, e_{n+1},\gamma_i,\gamma_{i+1})$ does not have superfluous intersections whenever $i\leq n$. Note that necessarily $\gamma_1=e_0$ and $\gamma_{n+1}=e_{n+1}$. For each $i=1,\dots,n$, set $c_i=\gamma_i*e_i* \bar{\gamma}_{i+1}$. Thus the image of $c_i$ comprises three geodesics whose complement in $V$ consists of an outer component containing $\partial V \setminus \partial X$ and at most one inner component. All possible topological types of the closed curve $c_i$ are shown in \Cref{fig:triangle_types}; see also Figure 17 on page 51 of \cite{AZ:67}. If $|c_i|$ bounds an inner component, let $T_i\subset V$ be the triangle bounded by the Jordan curve obtained by deleting the inward- and outward-pointing ends of $c_i$ when such exist. Otherwise, we set $T_i = \emptyset$. We now set $U = T_1 \cup \cdots \cup T_{n}$. Note that $\diam(U)\leq \varepsilon$, since $U\subset V$. 

\begin{figure} 
    \centering
    \begin{tikzpicture}[scale=1.4]
    \draw[black,thick] (0,1.4) to (1,1.4) to (.5,2.1) to (0,1.4);
    \filldraw [black] (0,1.4) circle (.7pt);
    \filldraw [black] (1,1.4) circle (.7pt);
    \filldraw [black] (.5,2.1) circle (.7pt);
    \draw[black,thick] (1.5,1.4) to (2.5,1.4) to (2,1.9) to (2,2.2) to (2,1.9) to (1.5,1.4);
    \filldraw [black] (1.5,1.4) circle (.7pt);
    \filldraw [black] (2.5,1.4) circle (.7pt);
    \filldraw [black] (2,2.2) circle (.7pt);
    \draw[black,thick] (2.9,1.3) to (3.2,1.5) to (3.8,1.5) to (4.1,1.3) to (3.8,1.5) to (3.5,2.1) to (3.2,1.5);
    \filldraw [black] (2.9,1.3) circle (.7pt);
    \filldraw [black] (4.1,1.3) circle (.7pt);
    \filldraw [black] (3.5,2.1) circle (.7pt);
    \draw[black,thick] (4.4,1.3) to (4.7,1.5) to (5.3,1.5) to (5.6,1.3) to (5.3,1.5) to (5,1.9) to (5,2.2) to (5,1.9) to (4.7,1.5);
    \filldraw [black] (4.4,1.3) circle (.7pt);
    \filldraw [black] (5.6,1.3) circle (.7pt);
    \filldraw [black] (5,2.2) circle (.7pt);
    \draw[black,thick] (6,1.3) to (6.5,1.7) to (7,1.3) to (6.5,1.7) to (6.5,2.2);
    \filldraw [black] (6,1.3) circle (.7pt);
    \filldraw [black] (7,1.3) circle (.7pt);
    \filldraw [black] (6.5,2.2) circle (.7pt);
    \draw[black,thick] (8,1.3) to (8,2.3);
    \filldraw [black] (8,1.3) circle (.7pt);
    \filldraw [black] (8,1.8) circle (.7pt);
    \filldraw [black] (8,2.3) circle (.7pt);
    \draw[black,thick] (0,0) to (1,0) to (.5,.7) to (.5,.3) to (.5,.7) to (0,0);
    \filldraw [black] (0,0) circle (.7pt);
    \filldraw [black] (1,0) circle (.7pt);
    \filldraw [black] (.5,.3) circle (.7pt);
    \draw[black,thick] (1.9,.2) to (1.5,0) to (2.5,0) to (2.1,.2) to (2.5,0) to (2,.7) to (1.5,0);
    \filldraw [black] (1.9,.2) circle (.7pt);
    \filldraw [black] (2.1,.2) circle (.7pt);
    \filldraw [black] (2,.7) circle (.7pt);
    \draw[black,thick] (3.4,.2) to (3,0) to (4,0) to (3.6,.2) to (4,0) to (3.5,.75) to (3.5,.35) to (3.5,.75) to (3,0);
    \filldraw [black] (3.4,.2) circle (.7pt);
    \filldraw [black] (3.6,.2) circle (.7pt);
    \filldraw [black] (3.5,.35) circle (.7pt);
    \draw[black,thick] (4.9,.15) to (4.5,0) to (5.5,0) to (5,.5) to (5,.8) to (5,.5) to (4.5,0);
    \filldraw [black] (4.9,.15) circle (.7pt);
    \filldraw [black] (5.5,0) circle (.7pt);
    \filldraw [black] (5,.8) circle (.7pt);
    \draw[black,thick] (5.9,-.1) to (6.2,.1) to (6.8,.1) to (7.1,-.1) to (6.8,.1) to (6.5,.7) to (6.5,.3) to (6.5,.7) to (6.2,.1);
    \filldraw [black] (5.9,-.1) circle (.7pt);
    \filldraw [black] (7.1,-.1) circle (.7pt);
    \filldraw [black] (6.5,.3) circle (.7pt);
    \draw[black,thick] (7.9,.2) to (7.5,0) to (8.5,0) to (8.1,.2) to (8.5,0) to (8,.5) to (8,.8) to (8,.5) to (7.5,0);
    \filldraw [black] (7.9,.2) circle (.7pt);
    \filldraw [black] (8.1,.2) circle (.7pt);
    \filldraw [black] (8,.8) circle (.7pt);
\end{tikzpicture}
    \caption{
    }
    \label{fig:triangle_types}
\end{figure}

It remains to check that $U$ is a neighbourhood of $x$. To do so, it suffices to show that $P\subset U$. Let $O=P^\circ \setminus \bigcup_{i=1}^{n+1} |\gamma_i|$. Then $O$ is a dense subset of $P$. Since the triangles are compact, the claim reduces to proving that $O\subset T_1 \cup \cdots \cup T_{n}$. To this end, consider a point $y\in O$. First, observe that $\partial P$ winds around $y$ within $U$. However, $\partial P= e_0*\cdots*e_{n+1}$ is path homotopic to 
\[e_0 *e_1 *\bar{\gamma}_2*\gamma_2 * e_2* \bar{\gamma}_3 *\gamma_3*e_3* \cdots *e_{n-1}* \bar{\gamma}_{n}*\gamma_{n}* e_{n}*e_{n+1}\]
within $V\setminus \{y\}$, which is in turn equal to $c_1*\cdots *c_n$. Thus, by additivity of the winding number, there must be some $i$ such that $c_i$ winds around $y$ within $U$. This implies that $y\in T_i$. 
\end{proof}

Note that, even for surfaces of bounded curvature, the proof of \Cref{lemm3} is shorter and conceptually simpler than that of \cite[Lem.~III.5, p.53]{AZ:67}.
\subsection{Covering by absolutely convex polygons}
\label{subsec:covabs}
At this point, \Cref{lemm3} gives a cover of $X$ by polygons of small perimeter. In this section, we improve this to a cover by absolutely convex polygons. 
\begin{prop}
\label{prop:abs}
Let $\varepsilon >0$. Then $X$ may be covered by a locally finite collection of absolutely convex polygons $(P_i)_{i\in \mathcal{I}}$ such that $\diam(P_i) \leq \varepsilon$ and $\partial P_i \setminus \partial X$ consists of transit points for each $i \in \mathcal{I}$.
\end{prop}
The proof relies on \Cref{lemm3} and the following lemma.

\begin{lemm}[cf. \cite{AZ:67}, Lemma III.4, p.51]
\label{lem:Abs}
Let $x\in X$ and $\varepsilon >0$. Then there is $\delta>0$ such that any polygon~$P$ with $x \in P$, $\diam(P)\leq \delta$ and $\ell(\partial P)\leq \delta$ is contained in an absolutely convex polygon $Q$ such that $\diam(Q) \leq \varepsilon$ and $\partial Q\setminus \partial P$ consists of transit points.
\end{lemm}
The idea of the proof is to take $Q$ to be a polygon of least perimeter among all polygons containing $P$ and contained in some fixed ambient disk $U$. This guarantees the boundary convexity of $Q$. Complete convexity is achieved by taking, among all polygons of least perimeter, the maximal one with respect to set inclusion.
\begin{proof}[Proof of \Cref{lem:Abs}]
We assume without loss of generality that $\varepsilon \leq \diam(X)/3$. Let $U$ be a disk neighbourhood of $x$ such that $\diam(U)\leq \varepsilon$ and set \[\zeta=d(x,\partial U\setminus \partial X)/5 >0.\] 
Let $0<\delta <\zeta$ be such that any Jordan curve $c$ in $U$ of length at most~$\delta$ which either winds around or passes through~$x$ is contained in $B(x,\zeta)$. 
Such a value $\delta$ must exist, since otherwise there would be a point separating $B(x,\zeta/2)$ from $\partial U \setminus \partial X$. Compare also the respective step in the proof of \Cref{lemm1}. 

Now let $P$ be a polygon with $x\in P$ such that $\ell(\partial P)\leq \delta$  and $\diam(P) \leq \delta$. Then by our choice of~$\delta$ and the diameter bound of $P$ we have $P\subset U$. By the Arzel\`a--Ascoli theorem, lower semicontinuity of length and continuity of the winding number, there exists a  closed curve $c$ that is shortest among all curves that wind around~$P^\circ$ within~$U$. Denote by $\mathcal{C}$ the collection of all such shortest curves $c$. By the additivity of the winding number, every $c\in \mathcal{C}$ must be a Jordan curve, and we denote its enclosed disk within~$U$ by~$Q_c$. 
Then $Q_c\subset U$ is a disk which contains $P$.

Next, we show that $Q_c$ is boundary convex with respect to $U$. Certainly $U$ is of diameter at most $\diam(X)/3$. Furthermore, since $\ell(\partial Q_c)\leq \delta$, we have that $\partial Q_c$ is contained in $B(x, \zeta)$ and hence that
\begin{equation}
\label{eq:xx}
d(Q_c,\partial U\setminus \partial X)\geq d(x,\partial U\setminus \partial X)-\zeta=4\zeta >4 \delta \geq 4 \cdot \ell(\partial P) \geq 4\cdot \ell(\partial Q_c).
\end{equation}
Now let $\gamma$ be a proper subcurve of $c$ and $\eta\subset U\setminus Q^\circ_c$ be a simple curve which is path homotopic to $\gamma$ within $U\setminus Q^\circ_c$. Then we must have that $\ell(\gamma)\leq \ell(\eta)$. Otherwise, the curve obtained from $c$ when replacing $\gamma$ with $\eta$ would be shorter than $c$ and non-contractible within $U\setminus P^\circ$. This would contradict to the assumption that $c\in \mathcal{C}$. We conclude that $Q_c$ is boundary convex with respect to~$U$.

We claim that every $c\in \mathcal{C}$ is piecewise geodesic and hence that the corresponding disk $Q_c$ is a polygon. Since $c$ is a compact curve, it suffices to show that $c$ is locally piecewise geodesic. By symmetry, we may furthermore assume that $c$ is parametrized by arc length on the interval $[-l,l]$ and only show that $c$ is locally piecewise geodesic at $0$. 
First, assume that $c (0)\in U\setminus \partial P$. In this case, the restriction of $c$ to $[-\mu,\mu]$ is a geodesic, where
\[
\mu=\frac{1}{8}\cdot \min \{ d(c(0),\partial P), d(c(0),\partial U\setminus \partial X),l\}>0.
\]
Otherwise since every geodesic from  $c(-\mu)$ to $c(\mu)$ must be contained in $U\setminus P^\circ$, we could  shorten $c$ within the admissible class by replacing one of the arcs of $c$ by the geodesic. Next, if $c(0)\in \partial P$, then let $\nu$ be the union of the (at most two) edges of $\partial P$ which contain $c(0)$. Then the restrictions of $c$ to $[-\mu,0]$ and $[0,\mu]$ are geodesics, where \[\mu=\frac{1}{2} \cdot \min \{d(c(0),\partial P \setminus \nu), d(c(0),\partial U\setminus \partial X), l\}>0.\] Namely, by our choice of $\mu$, every geodesic $\gamma$ from $c(0)$ to $c(-\mu)$ (respectively, to $c(\mu)$) is contained in $U$ and can intersect $\partial P$ only in $\nu$. Thus, by replacing a subcurve of $\gamma$ with a subarc of $P$, we can also find a geodesic $\eta$ from $c(0)$ to $c(-\mu)$ (respectively, to $c(\mu)$) which is contained in $U\setminus P^\circ$. Now, if the restriction of $c$ to $[-\mu,0]$ (respectively, to $[0,\mu]$) were not a geodesic, then we could, as before, shorten $c$ within the admissible class by replacing one of its arcs with $\eta$. Note that our argument also shows that $\partial Q_c \setminus \partial P$ consists of transit points.

It remains to show that we can find $c\in \mathcal{C}$ such that $Q_c$ is completely convex. We find such a curve $c$ as an application of Zorn's lemma. To do this, we introduce the partial ordering $\preceq$ on $\mathcal{C}$  defined by $c_1\preceq c_2$ whenever $Q_{c_1}^\circ\subset Q_{c_2}^\circ$. To apply Zorn's lemma, we must show that every chain $C\subset \mathcal{C}$ has an upper bound. Since $X$ is a second countable space, we may assume that $C = (c_i)_{i=1}^\infty$ with $c_1 \preceq c_2 \preceq c_3 \preceq \cdots$; see e.g.~\cite[Thm.~30.3]{Mun:00}. Then, by the Arzel\`a--Ascoli theorem, lower semicontinuity of length and continuity of the winding number, the sequence $(c_i)_{i=1}^\infty$ subconverges to some $c\in \mathcal{C}$, which must certainly be an upper bound for $C$ with respect to~$\preceq$. We conclude by Zorn's lemma that there is a maximal element $c_m\in \mathcal{C}$ with respect to~$\preceq$. We claim that $Q=Q_{c_m}$ is completely convex. Otherwise, there would be a geodesic $\gamma$ with endpoints in $Q$ which is not entirely contained in $Q$. By \eqref{eq:xx}, we would have $|\gamma|\subset U$ and, by passing to a subgeodesic, we could assume that $\gamma$ intersects $Q$ only in its endpoints. However, then we could enlarge $Q$ by replacing one of the arcs of $\partial Q$ with~$\gamma$. This would contradict the maximality of $c_m$.
\end{proof}

\begin{proof}[Proof of \Cref{prop:abs}]
We first assume that $X$ is compact. For each $x\in X$, choose~$\delta_x>0$ as in \Cref{lem:Abs} according to $x$ and $\varepsilon$. By \Cref{lemm3}, there is a neighbourhood $V_x$ of $x$ such that $\diam(V_x)\leq \delta_x/3$ and $V_x=T_x^1 \cup \cdots \cup T_x^{n_x}$, where the $T_x^i$ are triangles such that $\partial T_x^i \setminus \partial X$ consists of transit points. By deleting some of the triangles if necessary, we may assume that each $T_x^i$ contains $x$. Then, by \Cref{lem:Abs}, each triangle $T_x^i$ is contained in an absolutely convex polygon $P_x^i$  such that $\diam(P_x^i)\leq \varepsilon$ and $\partial P_x^i \setminus \partial X$ consists of transit points. Since $X$ is compact and the interiors of the neighbourhoods $V_x$ cover $X$, by choosing a subcollection of the $P_x^i$, we find our desired finite cover $(P_i)_{i=1}^k$.

If $X$ is non-compact, then we find a sequence $(\Omega_j)_{j=1}^\infty$ of relatively compact open sets such that $X=\bigcup_{j=1}^\infty \Omega_j$ and $\overline{\Omega}_j \subset \Omega_{j+1}$ for each $j$. For each $x \in X$, choose $0 < \varepsilon_x \leq \varepsilon$ so that $\overline{B}(x,\varepsilon_x)\subset \Omega_{j+1}\setminus \overline{\Omega}_{j-2}$ whenever $x\in \overline{\Omega_j}\setminus \Omega_{j-1}$. Now, we choose $\delta_x$ as in the compact case but according to $\varepsilon_x$ instead of $\varepsilon$. As before, we choose then the sets $V_x$, $T_x^i$ and $P_x^i$. Now for each $j$, by choosing a subcollection of the $P_x^i$, we can find finitely many absolutely convex polygon $P_1^j, \dots, P_{k_j}^j$ such that 
\[
\overline{\Omega}_j\setminus \Omega_{j-1} \subset P_1^j\cup \cdots \cup P_{k_j}^j \subset \subset \Omega_{j+1} \setminus \overline{\Omega}_{j-2}
\]
and each $P_i^j$ is of diameter at most $\varepsilon$ and such that $\partial P_i^j \setminus\partial X$ consists of transit points. Then $(P_i^j)_{j\in \mathbb{N}, 1\leq i \leq k_j}$ is locally finite and a cover of the desired type.
\end{proof}

\begin{rem}
\label{rem:nc}
When applying \Cref{thm:maingen} to non-compact surfaces, it might be helpful to have further control on the diameters of the triangles. Indeed, without serious additional difficulties, one can replace \ref{item:thm_ii} by the following stronger conclusion: Let $(\Omega_j)_{j=1}^\infty$ be an exhaustion of $X$ by relatively compact open sets as in the preceeding proof and $(\varepsilon_j)_{j=1}^\infty$ be a sequence of positive reals. Then every triangle~$T_i$ that intersects $X\setminus \overline{\Omega}_j$  has diameter at most $\varepsilon_j$. This is possible since the respective diameter bound can be achieved in \Cref{prop:abs}, and the latter steps will proceed by subdividing this given cover.  
\end{rem}

\subsection{Handling superfluous intersections} \label{subsec:covsup}

By \Cref{prop:abs}, we are now able to cover $X$ by small absolutely convex polygons. Our next objective is to find a cover by small \emph{non-overlapping} boundary convex polygons. The argument given in \cite{AZ:67} relies on a lemma stating that if $P_1, \ldots, P_n$ are non-overlapping boundary convex polygons and $P_{n+1}$ is absolutely convex, then one can subdivide $\bigcup_{j=1}^{n+1} P_j$ into finitely many non-overlapping boundary convex polygons. This is found as Lemma III.6 in \cite{AZ:67}. If the initial polygons $P_1, \ldots, P_{n+1}$ do not have superfluous intersections, then this claim is correct and the proof given in \cite{AZ:67} applies. However, if these polygons have superfluous intersections, the procedure given in the proof of Lemma III.6 in \cite{AZ:67} may not work, and it is not clear whether there exists a general procedure that remedies this. 

\begin{figure} 
    \centering
    \hfill
    \subfloat[]{
    \begin{tikzpicture}[scale=2.7]
    \draw[blue,thick] (0,0) to (2,0) to (1,.6) to (0,0);
    \draw[red,thick] (1.8,.12) to (1,.9) to (.2,.12) .. controls (.2,.1) and (.3,0) .. (.4,0) to (.44,0) .. controls (.45,.03) and (.48,.03) .. (.49,0) to (.53,0) .. controls (.56,.06) and (.64,.06) .. (.67,0) to (.71,0) .. controls (.72,.03) and (.75,.03) .. (.76,0) to (.8,0) .. controls (.9,.2) and (1.1,.2) .. (1.2,0) to (1.24,0) .. controls (1.25,.03) and (1.28,.03) .. (1.29,0) to (1.33,0) .. controls (1.36,.06) and (1.44,.06) .. (1.47,0) to (1.51,0) .. controls (1.52,.03) and (1.55,.03) .. (1.56,0) to (1.6,0) .. controls (1.7,0) and (1.8,.1) .. (1.8,.12); 
    \filldraw [red] (1,.9) circle (.5pt);
    \filldraw [red] (.2,.12) circle (.5pt);
    \filldraw [red] (1.8,.12) circle (.5pt);
    \filldraw [blue] (1,.6) circle (.5pt);
    \filldraw [blue] (0,0) circle (.5pt);
    \filldraw [blue] (2,0) circle (.5pt);
    \node[blue] at (1.95,.125) {\Large $P_2$};
    \node[red] at (1.15,.9) {\Large $P_1$};
    \end{tikzpicture}
    \label{fig:bad_intersection_1}}
    \hfill
    \subfloat[]{
    \begin{tikzpicture}[scale=2.7]
    \draw[blue,thick] (1.8,0) to (1,.6) to (.2,0) to (1.8,0);
    \draw[red,thick] (1.8,-.1) to (1,-.6) to (.2,-.1) .. controls (.25,-.05) and (.3,0) .. (.4,0) to (.44,0) .. controls (.45,.03) and (.48,.03) .. (.49,0) to (.53,0) .. controls (.56,.06) and (.64,.06) .. (.67,0) to (.71,0) .. controls (.72,.03) and (.75,.03) .. (.76,0) to (.8,0) .. controls (.9,.2) and (1.1,.2) .. (1.2,0) to (1.24,0) .. controls (1.25,.03) and (1.28,.03) .. (1.29,0) to (1.33,0) .. controls (1.36,.06) and (1.44,.06) .. (1.47,0) to (1.51,0) .. controls (1.52,.03) and (1.55,.03) .. (1.56,0) to (1.6,0) .. controls (1.7,0) and (1.75,-.05) .. (1.8,-.1); 
    \filldraw [red] (1,-.6) circle (.5pt);
    \filldraw [red] (.2,-.1) circle (.5pt);
    \filldraw [red] (1.8,-.1) circle (.5pt);
    \filldraw [blue] (1,.6) circle (.5pt);
    \filldraw [blue] (.2,0) circle (.5pt);
    \filldraw [blue] (1.8,0) circle (.5pt);
    \node[blue] at (1.,.4) {\Large $P_2$};
    \node[red] at (1.,-.4) {\Large $P_1$};
    \end{tikzpicture}
    \label{fig:bad_intersection_2}}
    \hfill \hfill
    \caption{}
    \label{fig:bad_intersection}
\end{figure}

\Cref{fig:bad_intersection_1} gives an example of a configuration consisting of a boundary convex polygon $P_1$ and absolutely convex polygon $P_2$ with edges intersecting in a Cantor set for which the procedure in \cite{AZ:67} does not apply. In \Cref{fig:bad_intersection_1}, $P_1$ cannot be enlarged without potentially interfering with its boundary convexity. Note, however, that in our situation we have the stronger property that each polygon $P_i$ is absolutely convex, not just boundary convex. This allows us to avoid configurations such as the one in \Cref{fig:bad_intersection_1}. On the other hand, a configuration such as that in \Cref{fig:bad_intersection_2} is possible for absolutely convex polygons and must be accounted for in the proof. These kinds of overlaps especially complicate the situation when attempting to decompose the union of more than two absolutely convex polygons.

Our argument proceeds in two steps. The first is the following lemma. It allows us to turn a cover by absolutely convex polygons into one by boundary convex polygons whose boundary edges form a locally finite topological graph.
\begin{lemm}
\label{lemm:abc}
Let $\mathcal{P}$ be a locally finite family of absolutely convex polygons in~$X$. Then there is a locally finite family of boundary convex polygons $\mathcal{Q}$ such that \begin{enumerate}
    \item \label{isup1}$\bigcup \mathcal{Q}=\bigcup \mathcal{P}$,
    \item \label{isup2} $\bigcup_{Q \in \mathcal{Q}} \partial Q$  
 is a locally finite topological graph,
 \item \label{isup0} each $Q\in \mathcal{Q}$ is contained in some $P\in \mathcal{P}$, and
 \item \label{isup4}$\bigcup_{Q\in \mathcal{Q}} \partial Q \setminus \bigcup_{P\in \mathcal{P}} \partial P$ consists of transit points.
\end{enumerate}
\end{lemm}
The idea of the proof is to enumerate the edges of the polygons in $\mathcal{P}$ and inductively replace them by means of \Cref{lemm:sup1} to obtain a locally finite graph. \Cref{lemm:part} then guarantees that the arising polygons are boundary convex.
\begin{proof}
We first consider the case of a finite collection $\mathcal{P}=\{P_1,\dots, P_k\}$ and prove the statement by induction on $k$. The base case $k=1$ is trivial. For the induction step, assume that the statement of the lemma holds for all collections of at most $k-1$ polygons. Consider a collection $\mathcal{P}=\{P_1,\dots,P_k\}$ of $k$ polygons as in the statement of the lemma. By the induction assumption, we may find a finite collection $\widehat{\mathcal{Q}}$ of boundary convex polygons satisfying properties (\ref{isup1})-(\ref{isup4}) for $\widehat{\mathcal{P}}=\{P_1,\dots, P_{k-1}\}$. Denote by $\widehat{\mathcal{E}}$ the finite graph formed by the the edges of the polygons $Q\in \widehat{\mathcal{Q}}$ and fix a disk~$U$ such that $P_k$ is absolutely convex with respect to $U$.

Denote by $(e_i)_{i=1}^m$ the edges of $P_k$. Applying Lemma~\ref{lemm:sup1} for each $i=1,\dots,m$ we can find a new geodesic $e_i^*$ such that 
\begin{enumerate}[label=(\roman*)]
    \item $e_i^*$ has the same endpoints as $e_i$, \label{item:4.4i}
    \item $|\widehat{\mathcal{E}}|\cup |e_i^*|$ is a finite topological graph, and \label{item:4.4ii}
    \item every connected component of $|e_i^*|\setminus |e_i|$ is contained in the image of some edge $e\in \widehat{\mathcal{E}}$. \label{item:4.4iii}
\end{enumerate}
By the absolute convexity of $P_k$, we must have $|e_i^*|\subset P_k$. Furthermore, by condition~\ref{item:4.4iii}, the interiors of the edges $e_i^*$ and $e_j^*$ for $i\neq j$ can only intersect within~$|\widehat{\mathcal{E}}|$. Thus, it follows from condition~\ref{item:4.4ii} that $|\widehat{\mathcal{E}}|\cup |e_1^*|\cup \cdots \cup |e_m^*|$ is also a finite topological graph.

Applying \Cref{lemm:part} inductively, it follows that the geodesics $e_1^*,\dots ,e_m^*$ subdivide $P_{k}$ into  a family of non-overlapping boundary convex polygons $\widetilde{\mathcal{Q}}_k$ such that $\bigcup_{Q\in \widetilde{\mathcal{Q}}_{k}} \partial Q \setminus \partial P_k$ consists of transit points. Let $\mathcal{Q}_{k}$ be the subcollection of those $Q\in \widetilde{\mathcal{Q}}_{k}$ which do not intersect $\partial P_{k} \setminus (|e_1^*| \cup \cdots \cup |e_m^*|)$. Note that $\mathcal{Q}_k$ is finite since $|e_1^*|\cup \cdots \cup |e_m^*|$ is a finite topological graph. Set $\mathcal{Q}=\widehat{\mathcal{Q}}\cup \mathcal{Q}_k$. It follows readily that the properties (\ref{isup2})-(\ref{isup4}) are satisfied for $\mathcal{Q}$. Since $\bigcup \widehat{\mathcal{Q}}=P_1\cup \cdots \cup P_{k-1}$ and $\bigcup \widetilde{\mathcal{Q}}_k=P_k$, to verify (\ref{isup1}) it suffices to show that every $Q\in \widetilde{\mathcal{Q}}_k\setminus \mathcal{Q}_k$ is already contained in some $P_j$ with $j<k$.

To this end, let $Q\in \widetilde{\mathcal{Q}}_{k} \setminus \mathcal{Q}_{k}$. Then $Q$ intersects $|e_i| \setminus |e_i^*|$ for some $1\leq i\leq m$. By \Cref{lemm:part}, $e_i^*$ subdivides $P_{k}$ into a family of nonverlapping polygons~$\mathcal{R}$. Certainly there must be some $R\in \mathcal{R}$ with $Q\subset R$. The boundary of $R$ is composed of a subcurve of $e_i$ together with a connected component of $|e_i^*|\setminus |e_i|$, denoted by $e_i'$. By \ref{item:4.4iii}, $e_i'$ is a subcurve of some edge of a polygon $\widehat{Q}\in \widehat{\mathcal{Q}}$, and hence, by property~\eqref{isup0} for~$\widehat{\mathcal{Q}}$, $e_i'$~is contained in some $P_j$ with $j<k$. Since $P_k$ and $P_j$ are completely convex, we have $\partial R\subset P_k\cap P_j$. Since $\partial R$ is a Jordan curve, it must bound some disk within $P_j$. If this disk were not equal to $R$, then we would have $X=R\cup P_j$ and hence $X=P_k\cup P_j$. This would be a contradiction since $X$ is connected and $P_j$ and $P_k$ are both of diameter at most $\diam (X)/3$. Thus we conclude that $R\subset P_j$ and hence also that $Q\subset P_j$.

In the case of an infinite collection $\mathcal{P} =\{P_1,P_2,P_3,\dots\}$, inductively apply the above construction setting $\mathcal{Q}=\mathcal{Q}_1\cup \mathcal{Q}_2 \cup \cdots$. Note here that, throughout the process, we only add polygons, and the previously constructed polygons remain unchanged. Thus, certainly $\mathcal{Q}$ satisfies \eqref{isup1} as well as \eqref{isup0} and \eqref{isup4}. Within any compact set, by the local finiteness of $\mathcal{P}$ and since $\bigcup \mathcal{Q}_k \subset P_k$, only finitely many steps are of interest. Hence the local finiteness of $\mathcal{Q}$ and of the graph formed by its edges also follow.
\end{proof}
\subsection{Covering by non-overlapping polygons}\label{subsec:covover} By the previous results, we are able to cover $X$ by small boundary convex polygons such that the boundary edges of the polygons form a locally finite graph. Using the following lemma, we can improve this to a cover of $X$ by small non-overlapping boundary convex polygon. 
\begin{lemm}
\label{prop:over}
Let $\mathcal{P}$ be a locally finite family of boundary convex polygons, and let $\Gamma \subset X$ be a locally finite topological graph with $\bigcup_{P\in \mathcal{P}}\partial P \subset \Gamma$.  Then there is a locally finite family of non-overlapping boundary convex polygons $\mathcal{Q}$ such that \begin{enumerate}
    \item \label{i_dec1}$\bigcup \mathcal{Q}=\bigcup \mathcal{P}$,
     \item $\Gamma \cup \bigcup_{Q\in \mathcal{Q}} \partial Q$ is a locally finite topological graph, \label{i_dec3}
    \item each $Q\in \mathcal{Q}$ is contained in some $P\in \mathcal{P}$, and\label{i_dec4}
    \item $\bigcup_{Q\in \mathcal{Q}} \partial Q \setminus \Gamma$ consists of transit points.\label{i_dec5}
\end{enumerate}

\end{lemm}

The proof of \Cref{thm:maingen} requires only the case where $\Gamma = \bigcup_{P \in \mathcal{P}} \partial P$. However, for technical reasons, it seems difficult to prove this special case directly.


To prove \Cref{prop:over}, we follow the idea suggested by the proof of Lemma~III.6 in~\cite{AZ:67}, which works provided that one assumes that the union of the boundary edges of the polygons form a finite graph. By \Cref{lemm:abc}, one can make this additional assumption. However, when doing so, one can no longer guarantee that each of the polygons is absolutely convex, and hence the intersection of two polygons $P_i, P_j \in \mathcal{P}$ may fail to be convex. Instead, to adapt the argument in \cite{AZ:67}, we employ \Cref{lemm:intersect} above. The decomposition procedure used to prove \Cref{prop:over} is illustrated in \Cref{fig:decomposing} for the case of two polygons $P_1, P_2$.
\begin{figure} 
    \hfill
    \subfloat{
    \begin{tikzpicture}[scale=2.8]
    \fill[gray,opacity=.2] (.333,.444) to (.5,.667) to (.75,.667) to (.75,0) to (.333,0);
    \fill[gray,opacity=.2] (1.1,.667) to (1.333,.667) to (1.333,.333) to (1.1,.333);
    \draw[thick,blue] (1.5,.333) to (1.5,.9) to (.333,.9) to (.333,0) .. controls (.333,-.2) and (.75,-.2) .. (.75,0) to (.75,.667) .. controls (.75,.8) and (1.1, .8) .. (1.1,.667) to (1.1,.333) to (1.5,.333);
    \draw[thick,red] (0,0) to (1.333,0) to (1.333,.667) to (.5,.667) to (0,0);
    \filldraw[blue] (1.5,.333) circle (.5pt);
    \filldraw[blue] (1.5,.9) circle (.5pt);
    \filldraw[blue] (.333,.9) circle (.5pt);
    \filldraw[blue] (.333,.1) circle (.5pt);
    \filldraw[blue] (.75,.1) circle (.5pt);
    \filldraw[blue] (.75,.6) circle (.5pt);
    \filldraw[blue] (1.1,.6) circle (.5pt);
    \filldraw[blue] (1.1,.333) circle (.5pt);
    \filldraw[blue] (1.5,.333) circle (.5pt);
    \filldraw[red] (0,0) circle (.5pt);
    \filldraw[red] (1.333,0) circle (.5pt);
    \filldraw[red] (1.333,.667) circle (.5pt);
    \filldraw[red] (.5,.667) circle (.5pt);
    \node[blue] at (.5,1.) {$P_1$};
    \node[red] at (0.05,.25) {$P_2$};
    \node[] at (0.55,.15) {$W_{11}$};
    \node[] at (1.21,.5) {$W_{12}$};
    \end{tikzpicture}
    } \hfill
    \subfloat{
    \begin{tikzpicture}[scale=2.8]
    \draw[thick,blue] (1.333,.333) to (1.5,.333) to (1.5,.9) to (.333,.9) to (.333,.333);
    \draw[thick,blue] (.333,0) .. controls (.333,-.2) and (.75,-.2) .. (.75,0); 
    \draw[thick,blue] (.75,.667) .. controls (.75,.8) and (1.1, .8) .. (1.1,.667);
    \draw[thick,red] (.333,.444) to (0,0) to (.333,0); 
    \draw[thick,red] (.75,0) to (1.333,0) to (1.333,.333);
    \draw[thick,red] (1.1,.667) to (.75,.667);
    \draw[violet,line width=0.6mm] (.333,.444) to (.333,0) to (.75,.0) to (.75,.667) to (.333,.444);
    \draw[violet,line width=0.6mm] (1.1,.667) to (1.333,.333);
    \filldraw[blue] (1.5,.333) circle (.5pt);
    \filldraw[blue] (1.5,.9) circle (.5pt);
    \filldraw[blue] (.333,.9) circle (.5pt);
    \filldraw[blue] (1.5,.333) circle (.5pt);
    \filldraw[red] (0,0) circle (.5pt);
    \filldraw[red] (1.333,0) circle (.5pt);
    \filldraw[violet] (1.333,.333) circle (.6pt);
    \filldraw[violet] (1.1,.667) circle (.6pt);
    \filldraw[violet] (.75,.667) circle (.6pt);
    \filldraw[violet] (.75,.0) circle (.6pt);
    \filldraw[violet] (.333,.0) circle (.6pt);
    \filldraw[violet] (.333,.444) circle (.6pt);
    \node[violet] at (0.275,-.1) {$v_{11}^0$};
    \node[violet] at (0.85,-.1) {$v_{11}^1$};
    \node[violet] at (0.875,.575) {$v_{11}^2$};
    \node[violet] at (0.225,.475) {$v_{11}^3$};
    \node[violet] at (1.2255,.675) {$v_{12}^1$};
    \node[violet] at (1.225,.275) {$v_{12}^0$};
    \end{tikzpicture}
    }
    \hfill \hfill
    \caption{}
    \label{fig:decomposing}
\end{figure}
\begin{proof}

We first consider the case of a finite collection $\mathcal{P}=\{P_1,\dots, P_n\}$. 
We may further assume that the polygons $P_1,\dots, P_{n-1}$ are non-overlapping.
The latter is possible since the case of a general finite collection follows from this special case by induction.

By \Cref{lemm:intersect}, and since $\bigcup_{P\in \mathcal{P}} \partial P$ is a finite topological graph, for each $i=1,\ldots, n$ the closure of the interior of $P_i\cap P_{n}$ is a union of finitely many non-overlapping boundary convex polygons $W_{i1},\dots, W_{ik_i}$. For each $j=1,\dots,k_i$, we let $v_{ij}^0,\dots ,v_{ij}^{m_{ij}}$  be a cyclic enumeration of the topological vertices of $\partial P_i\cup \partial P_{n}$ that lie in $\partial W_{ij}$.  Note that $W_{ij}$ is a geodesic surface with piecewise geodesic boundary and $\Gamma \cap W_{ij}$ is a finite graph. Thus, by iterated  application of \Cref{lemm:sup1}, we may find for each $1\leq i \leq n$, $1\leq j \leq k_i$ and $0\leq l\leq m_{ij}$ a geodesic $\gamma_{ij}^l \subset W_{ij}$ from $v_{ij}^l$ to $v_{ij}^{l+1}$ such that $\Gamma \cup \bigcup_{i=1}^n\bigcup_{j=1}^{k_i}\bigcup_{l=0}^{m_{ij}} |\gamma_{ij}^l|$ is a locally finite topological graph. Let $G$ be the finite topological graph that one obtains by deleting from $\partial P_1 \cup \dots\cup \partial P_n$ all topological edges which are contained in the interior of some $P_i$ and adding in all the geodesics $\gamma_{ij}^l$.  Now we set $Q_1,\dots, Q_r$ to be the closures of the complementary components of $G$ in $\bigcup_{i=1}^{n} P_i$ and take $\mathcal{Q}:=\{Q_1, \dots, Q_r\}$.

Then certainly the $Q_i$ are non-overlapping and form a decomposition of $\bigcup_{i=1}^{n} P_i$. Furthermore, since each component of $\partial P_i \setminus \bigcup_{l=0}^{m_{ij}} |\gamma_{ij}^l|$ in $W_{ij}$ is separated from $\partial P_n$ by $\bigcup_{l=0}^{m_{ij}} |\gamma_{ij}^l|$, and similarly for $\partial P_n$, it follows that the boundary of each complementary component of $\partial P_i\cup \partial P_n \cup  \bigcup_{l=0}^{m_{ij}} |\gamma_{ij}^l|$ in $W_{ij}$ is contained either in $\partial P_i \cup  \bigcup_{l=0}^{m_{ij}} |\gamma_{ij}^l|$ or contained in $\partial P_{n}  \cup \bigcup_{l=0}^{m_{ij}} |\gamma_{ij}^l|$. Thus we deduce that \eqref{i_dec4} holds. After noticing the latter, iterated application of \Cref{lemm:part} shows that each $Q_i$ is a boundary convex polygon. Finally, \eqref{i_dec3} and \eqref{i_dec5} hold, since \[
\bigcup_{i=1}^{n} \partial P_i \cup \bigcup_{i=1}^r \partial Q_i= \bigcup_{i=1}^{n} \partial P_i \cup \bigcup_{i=1}^n\bigcup_{j=1}^{k_i}\bigcup_{l=0}^{m_{ij}} |\gamma_{ij}^l|
\]
and the curves $\gamma_{ij}^l$ are geodesics with endpoints in $\bigcup_{i=1}^{n} \partial P_i \subset \Gamma$.

Now consider the case of a locally finite collection $\mathcal{P}= \{P_1,P_2,\dots \}$. We claim that for each $j\in \mathbb{N}$ there is a finite collection $\mathcal{Q}_j$ that satisfies the conclusions of the lemma for $\mathcal{P}_j := \{P_1, \ldots, P_j\}$, with the additional property that
   \[
   \{ Q\in \mathcal{Q}_j: Q\subset P_i\}=\{Q\in \mathcal{Q}_{j-1}: Q \subset P_i\}
   \]
whenever $i<j$ and the polygons $P_i$ and $P_j$ are non-overlapping. For the case $j=1$, we simply set $\mathcal{Q}_1:=\{Q_1\}$. Next, assume the family $\mathcal{Q}_j$ has been constructed. 
Denote by $\widetilde{\mathcal{Q}}_j$ the subfamily of those $Q\in \mathcal{Q}_j$ that overlap with $P_{j+1}$. The family $\mathcal{Q}_{j+1}$ is obtained from $\mathcal{Q}_j$ by replacing $\widetilde{\mathcal{Q}}_j$ with the family of those polygons obtained when applying the finite case of the result to $\widetilde{\mathcal{Q}}_j\cup \{P_{j+1}\}$. It is not hard to check that the desired properties carry over.


Set $\mathcal{Q}:=\bigcap_{j=1}^\infty \bigcup_{i=j}^\infty \mathcal{Q}_i$.  
Certainly, $\mathcal{Q}$ is a family of non-overlapping boundary convex polygons which satisfies \eqref{i_dec4} and \eqref{i_dec5}. Now let $K\subset X$ be compact. By the local finiteness  of $\mathcal{P}$, there is a maximal $l\in \mathbb{N}$ such that $P_l$ intersects $K$. 
Also by the local finiteness of $\mathcal{P}$, there is a maximal $s\in \mathbb{N}$ such that $P_s$ intersects $\bigcup \mathcal{P}_l$. Then the collection of those $Q\in \mathcal{Q}$ that intersect $K$ is the same as the collection of those $Q \in \mathcal{Q}_s$ that intersect $K$, and hence finite. We conclude that $\mathcal{Q}$ is locally finite. Similar arguments show \eqref{i_dec1} and \eqref{i_dec3}.
\end{proof}

\subsection{Decomposing polygons into triangles} 
\label{subsec:covtria}

After covering $X$ by a locally finite collection of small non-overlapping boundary convex polygons, the final step is to cut these polygons into triangles. This is achieved by the following lemma. Compare also \cite[p.60]{AZ:67}.
\begin{lemm}
\label{lemm:4}
Let $P\subset X$ be a boundary convex polygon. Then one can decompose~$P$ into non-overlapping boundary convex triangles $T_1, \dots , T_n$ such that $\partial T_i \setminus \partial P$ consists of transit points for each~$i$.
\end{lemm}
\begin{proof}
We prove this by induction on the number of vertices $k$ of $P$. The base cases $k=2$ and $k=3$ are trivial. Now assume the claim holds for polygons with at most $k-1$ vertices, where $k\geq 4$. Let $P$ be a boundary convex polygon with $k\geq 4$ vertices. Choose non-consecutive vertices $v$ and $w$ of  $\partial P$. Since $P$ is convex and has piecewise geodesic boundary, we may apply~\Cref{lemm:sup2} with $X$ replaced by $P$. Thus we can find a geodesic~$\gamma$ from $v$ to $w$ within $P$ which does not have superfluous intersections with $\partial P$. By \Cref{lemm:part} the curve~$\gamma$ subdivides~$P$ into finitely many non-overlapping boundary convex polygons $P_1,\dots,P_m$. The boundary of each $P_i$ is composed of a subgeodesic $\gamma_i$ of $\gamma$ and a subarc $\eta_i$ of $\partial P$. Since $v$ and $w$ are nonconsecutive we can arrange that on the interior of $\eta_i$ there lie at most $k-3$ vertices of $\partial P$. Thus we may represent each $P_i$ as a polygon with at most $k-1$ vertices. Furthermore, note that $\partial P_i\setminus \partial P$ consists of transit points. Thus we can derive the claim by applying the inductive assumption to each of the polygons~$P_i$.
\end{proof}
\Cref{prop:abs}, \Cref{lemm:abc}, \Cref{prop:over}  and \Cref{lemm:4} together complete the proof of \Cref{thm:maingen}, except for the non-degeneracy conclusion~\ref{item:non_degenerate}.
\section{ Decomposing bigons into non-degenerate triangles}
\label{sec:degenerate_triangles}

Let $X$ be a length surface as in \Cref{thm:maingen}. In \Cref{sec:proof}, we found a cover $(T_i)_{i \in \mathcal{I}}$ of $X$ by triangles satisfying all the conclusions of \Cref{thm:maingen} with the possible exception of property~\ref{item:non_degenerate}, namely that each triangle is non-degenerate. To complete the proof of \Cref{thm:maingen},  we apply the following proposition to each degenerate triangle.


\begin{prop} 
\label{lemm:deg_triang}
Let $B\subset X$ be a boundary convex triangle. Then $B$ may be decomposed into finitely many non-overlapping non-degenerate boundary convex triangles $T_1,\dots, T_n$ such that $\partial T_i \setminus \partial B$ consists of transit points for each $i$.
\end{prop}
\Cref{lemm:deg_triang} is proved for surfaces of bounded curvature~$X$ in \cite[Lemma III.7, p.60]{AZ:67}, and the proof of this special case is easier. The remainder of this section is dedicated to the proof of \Cref{lemm:deg_triang}.

\subsection{Preliminary remarks} \label{sec:non_degenerate_remarks}

The conclusion is trivial if $B$ is a non-degenerate triangle.  Thus, to prove \Cref{lemm:deg_triang}, we may assume that $B$ is a bigon. We call a boundary convex bigon $B\subset X$ \emph{indecomposable} if no subdivision as in \Cref{lemm:deg_triang} is possible. The strategy is to analyze the structure of a hypothetical indecomposable bigon until we reach a contradiction. Before giving the proof, we make several preliminary observations.

Consider a bigon $B$ with bottom vertex $b$, top vertex $t$, left side $L$ and right side $R$ as shown in \Cref{fig:bigon}. Assume that the curves $L$ and $R$ are parametrized by arc length on the interval $[0,a]$, beginning at $b$ and ending at $t$. We refer to $(b,t,L,R,a)$ as the \emph{data} associated with $B$. Note that, by the vertex perturbation trick described in \Cref{sec:disks}, if a bigon $B$ is indecomposable, then $\partial B$ is not locally a geodesic at each of its vertices.

\begin{figure} 
    \subfloat[]{
    \begin{tikzpicture}[scale=2.2]
    \draw[thick] (0,-.866) arc[start angle=-60, end angle=60, radius=1];
    \draw[thick] (0,.866) arc[start angle=120, end angle=240, radius=1];
    \filldraw (0,-.866) circle (.5pt);
    \filldraw (0,.866) circle (.5pt);
    \node at (.05,-.95) {$b$};
    \node at (.05,.95) {$t$};
    \node at (-.625,0) {$L$};
    \node at (.625,0) {$R$};
    \end{tikzpicture}
    \label{fig:bigon}}
    \hfill
    \subfloat[]{
    \begin{tikzpicture}[scale=2.2]
    \fill[gray,opacity=.1] (0,-.866) arc[start angle=-60, end angle=5, radius=1] to (-.366,-.5) arc[start angle=210, end angle=240, radius=1];
    \draw[thick] (0,-.866) arc[start angle=-60, end angle=60, radius=1];
    \draw[thick] (0,.866) arc[start angle=120, end angle=240, radius=1];
    \draw[blue,line width=0.5mm] (-.4658,.2588) arc[start angle=165, end angle=210, radius=1] to (.496,.087) arc[start angle=5, end angle=15, radius=1];
    \draw[red,line width=.5mm] (-.366,-.5) to (.05,-.2167) to (.496,-.087);
    \filldraw (0,-.866) circle (.5pt);
    \filldraw (0,.866) circle (.5pt);
    \filldraw[blue] (-.4658,.2588) circle (.5pt);
    \filldraw[blue] (.4658,.2588) circle (.5pt);
    \filldraw[red] (-.366,-.5) circle (.5pt);
    \filldraw[blue] (.496,.087) circle (.5pt);
    \node at (.05,-.95) {$b=b'$};
    \node at (.05,.95) {$t$};
    \node at (-.685,.1825) {$L(s)$};
    \node at (.675,.275) {$R(s)$};
    \node at (.6,.075) {$t'$};
    \node at (-.62,-.525) {$L(m')$};
    \node[blue] at (-.205,-.275) {\Large $\gamma$};
    \node at (0,-.575) {\Large $B'$};
    \node[red] at (.2,-.3) {\Large $\alpha$};
    \end{tikzpicture}
    \label{fig:bigon_2}} \hfill
    \subfloat[]{
    \begin{tikzpicture}[scale=2.2]
    \draw[thick] (0,-.866) arc[start angle=-60, end angle=60, radius=1];
    \draw[thick] (0,.866) arc[start angle=120, end angle=240, radius=1];
    \draw[blue,line width=0.5mm] (-.496,.087) arc[start angle=175, end angle=190, radius=1] to (.2071,.7071);
    \draw[red,line width=.5mm] (.25,.55) to (-.319,.574) arc[start angle=145, end angle=155, radius=1] to (-.2,-.1);
    \filldraw (0,-.866) circle (.5pt);
    \filldraw (0,.866) circle (.5pt);
    \filldraw[blue] (-.496,.087) circle (.5pt);
    \filldraw[blue] (.2071,.7071) circle (.5pt);
    \node at (.05,-.95) {$b$};
    \node at (.05,.95) {$t$};
    \node at (-.725,.1) {$L(m)$};
    \node at (.45,.71) {$R(r_n)$};
    \node[blue] at (-.025,.2) {\Large $\alpha_n$};
    \node[red] at (-.025,.65) {\Large $\gamma$};
    \end{tikzpicture}
    \label{fig:bigon_3}}
    \caption{}
    \label{fig:Prop(ii)}
\end{figure}

By \Cref{rem:sup_big}, any pair of points $L(l),R(r)$ with $l,r\in (0,a)$ is joined by a geodesic that does not have superfluous intersections with $L$ and $R$. Such a geodesic $\gamma$ will be called \emph{horizontal}, and we denote $l_\gamma=l$ and $r_\gamma=r$. We say that a horizontal geodesic is \emph{transverse} if it intersects $\partial B$ only in its endpoints. From the previous paragraph, we see that any indecomposable bigon $B$ contains an abundance of transverse geodesics. Namely, every horizontal geodesic $\gamma$ contains a unique transverse subcurve, denoted by $\widehat{\gamma}$, which we refer to as the \textit{transverse part} of $\gamma$. We say that $\gamma$ \textit{points upward} if $l_{\widehat{\gamma}}<r_{\widehat{\gamma}}$ and \textit{points downward} if $r_{\widehat{\gamma}}<l_{\widehat{\gamma}}$.

By \Cref{lemm:part}, a transverse geodesic $\gamma$ splits $B$ into two boundary convex triangles $B_b^\gamma$ and $B_t^\gamma$ with respective vertex sets $\{b,L(l_{\gamma}),R(r_\gamma)\}$ and $\{t,L(l_{\gamma}),R(r_\gamma)\}$ such that  $(\partial B_b^\gamma \cup \partial B_t^\gamma )\setminus \partial B$ consists of transit points. We claim that if $B$ is an indecomposable bigon, then both $B_t^\gamma$ and $B_b^\gamma$ can be represented as bigons. Indeed, if $B$ is indecomposable, then at least one of $B_t^\gamma$ and $B_b^\gamma$ must be representable as an indecomposable bigon. Assume without loss of generality that this is the case for $B_b^\gamma$. Since $B_b^\gamma$ is indecomposable, the representation of $B_b^\gamma$ as a triangle with vertices $b,L(l_{\gamma}),R(r_\gamma)$ must reduce to a bigon. However, we cannot remove $b$, since $\partial B$ is not locally a geodesic at $b$. Hence we conclude that $\ell(\gamma)=|r_\gamma-l_\gamma|$, and in particular that $B_t^\gamma$ is also a bigon.

If $\gamma_1, \gamma_2$ are two horizontal geodesics, we say that $\gamma_1$ is \textit{below} (respectively, \emph{above})~$\gamma_2$ if $|\widehat{\gamma}_1|$ is contained in $B_b^{\widehat{\gamma}_2}$  (respectively, in $B_t^{\widehat{\gamma}_2}$). Given $l,r\in (0,a)$ we can find a bottommost and a topmost geodesic among all horizontal geodesics $\gamma$ with $l_\gamma=l$ and $r_\gamma=r$. 
To construct these curves, one proceeds as in the construction of the outermost curve in the proof of \Cref{lem:Abs}.
\subsection{Proof of \Cref{lemm:deg_triang}}

Let $B\subset X$ be an indecomposable bigon with data $(b,t,L,R,a)$. As the first step, we show that we can replace $B$ by another indecomposable bigon in a way that yields additional information on certain geodesics. See \ref{item:degtriangle1} and \ref{item:degtriangle2} below.

Fix $s\in (0,a)$. Let $\gamma$ be a horizontal geodesic with $l_\gamma=r_\gamma =s$. Then we must either have that 
\[l_{\widehat{\gamma}}\leq s \ \ \text{and}  \ \ r_{\widehat{\gamma}}\leq s,
\] or that \[l_{\widehat{\gamma}}\geq s \ \ \text{and}  \ \ r_{\widehat{\gamma}}\geq s.
\]
Otherwise, $\widehat{\gamma}$ would split $B$ in such a way that both $B^{\widehat{\gamma}}_t$ and $B^{\widehat{\gamma}}_b$ are decomposable using the vertex perturbation trick. By interchanging $b$ and $t$ and reorienting $L$ and $R$ if needed, we may assume that $l_{\widehat{\gamma}}\leq s$ and $r_{\widehat{\gamma}}\leq s$.

Now choose $\gamma$ to be bottommost among all horizontal geodesics with $l_\gamma=s$ and $r_\gamma=s$. After possibly interchanging $L$ and $R$, we may also assume that $\widehat{\gamma}$ points upward. Then $B_t^{\widehat{\gamma}}$ can be decomposed by means of the vertex perturbation trick, and hence $B_b^{\widehat{\gamma}}$ is indecomposable. Thus $B'= B_b^{\widehat{\gamma}}$ is a new indecomposable bigon with data $(b',t',L',R',a')$, where $b'=b$, $t'=R(r_{\widehat{\gamma}})$, $L'$ is the composition of the restriction of $L$ to $[0,l_{\widehat{\gamma}}]$ and $\widehat{\gamma}$, $R'$ is the restriction of $R$ to 
$[0,r_{\widehat{\gamma}}]$, and $a'=r_{\widehat{\gamma}}$. Additionally, for $m'=l_{\widehat{\gamma}}\in (0,a')$ and the indecomposable bigon $B'$, we have the following properties:
\begin{enumerate}[label=(\roman*)]
    \item \label{item:degtriangle1}
    The restriction of $L'$ to $[m',a']$ is the unique geodesic within $B'$ from $L(m')$ to $t'$, and
    \item \label{item:degtriangle2} every horizontal geodesic $\alpha\subset B'$ with $l_\alpha=m'$ and $r_{\widehat{\alpha}}\in (m',a)$ must satisfy $l_{\widehat{\alpha}}<m'$.
\end{enumerate}
Property \ref{item:degtriangle1} follows immediately from our assumption that $\gamma$ is downmost between its endpoints. For the second property, assume that $\alpha\subset B'$ is a horizontal geodesic with $m'=l_\alpha\leq l_{\widehat{\alpha}}$ and $r_{\widehat{\alpha}}\in (m',a')$. We may assume without loss of generality that $r_\alpha=r_{\widehat{\alpha}}$. Then $\alpha$ is a transverse geodesic when considered not with respect to $B'$ but with respect to $B$. See \Cref{fig:bigon_2}. In particular, we must have that $\ell(\alpha)=r_\alpha-m'$. Thus we obtain a geodesic within $B'$ from $L(m')$ to $t'$ by composing $\alpha$ with the subarc of $R'$ from $R'(r_\alpha)$ to $t$. This contradicts \ref{item:degtriangle1}.

From now on, by replacing $B$ with $B'$, we may assume that our indecomposable bigon has the properties \ref{item:degtriangle1} and \ref{item:degtriangle2} for some $m\in (0,a)$. 

\begin{lemm}
\label{lemm:contra}
Let $B$ be an indecomposable bigon with data $(b,t,L,R,a)$ that satisfies properties \ref{item:degtriangle1} and \ref{item:degtriangle2} for some $m\in (0,a)$. Then any geodesic $\gamma \subset B$ that intersects $L([m,a])$ must have at least one endpoint in $L([m,a])$.
\end{lemm}
\begin{proof}
Assume that $\gamma\subset B$ is a geodesic that intersects $L([m,a])$ and has both endpoints in $B\setminus L([m,a])$.

Let $(r_n)$ be a sequence satisfying $r_n\nearrow a$. For each $n$, let $\alpha_n \subset B$ be a horizontal geodesic with $l_{\alpha_n}=m$ and $r_{\alpha_n}=r_n$. Then, by the Arzel\`a--Ascoli theorem and~\ref{item:degtriangle1}, the sequence $(\alpha_n)$ must converge uniformly to the restriction of $L$ to $[m,a]$. So we must eventually have that $r_{\widehat{\alpha}_n}>m$ and hence by \ref{item:degtriangle2} that $l_{\widehat{\alpha}_n}<m$. By this observation, property \ref{item:degtriangle2} and the fact that the $\alpha_n$ uniformly converge to the restriction of $L$ to $[m,a]$, we see that~$\alpha_n$ must intersect $\gamma$ before and after $\gamma$ intersects $L([m,a])$ for all sufficiently large $n$. See \Cref{fig:bigon_3}.  Thus, choosing such $n$ and replacing some portion of $\alpha_n$ by the respective one of $\gamma$, we obtain a geodesic $\beta$ from $L(m)$ to $R(r_n)$ which initially moves downward along $L$ but then on its way intersects $L((m,a])$. However, such $\beta$ cannot exist by \ref{item:degtriangle1}.
\end{proof}

The remaining goal is to construct a curve $\gamma$ as in the statement of \Cref{lemm:contra} and hence reach a contradiction.  Before doing so, we need the following additional observation.

\begin{lemm}
\label{lemm:ext}
Let $B$ be an indecomposable bigon with data $(b,t,L,R,a)$ that satisfies properties \ref{item:degtriangle1} and \ref{item:degtriangle2} for some $m\in (0,a)$. Then there is a horizontal geodesic $\gamma\subset B$ with $l_{\widehat{\gamma}}<m<l_\gamma$.
\end{lemm}
\begin{proof}
Reasoning as in the proof of \Cref{lemm:contra}, we find a horizontal geodesic $\alpha$ such that $l_{\widehat{\alpha}}< m=l_\alpha$ and $r:=r_\alpha=r_{\widehat{\alpha}}\in (m,a)$. For each $s\in (m,r)$, let $\gamma_s$ be a horizontal geodesic with $l_{\gamma_s}=s$ and $r_{\gamma_s}=r$. By replacing a subcurve of $\gamma_s$ with a subcurve of $\widehat{\alpha}$ if needed, we may assume that $r_{\widehat{\gamma}_s}\geq r>m$. If $l_{\widehat{\gamma}_s}<m$ for some $s$, then we may set $\gamma=\gamma_s$.

So suppose that $l_{\widehat{\gamma}_s}\geq m$ for all $s\in (m,r)$. Then $\widehat{\gamma}_s$ must point downward for each $s$. Otherwise, we would obtain a contradiction to \ref{item:degtriangle1} by considering the curve which goes along $L$ from $L(m)$ to $L(l_{\widehat{\gamma}_s})$, then along $\widehat{\gamma}_s$ and then along $R$ from $R(r_{\widehat{\gamma}_s})$ to $t$. In particular, we must have $l_{\widehat{\gamma}_s}>r$ and hence that $L([s,r])\subset |\gamma_s|$. See also \Cref{fig:more_nondegenerate}. Thus, by the Arzel\`a--Ascoli theorem, we can find a geodesic $\eta$ from $L(m)$ to $R(r)$ with $L([m,r])\subset |\eta|$. Again, we could assume that $\eta$ is horizontal and that $r_{\widehat{\eta}} \geq r>m$. However, then $\eta$ would contradict \ref{item:degtriangle2}.
\end{proof}

\begin{figure}
    \centering
    \begin{tikzpicture}[scale=2.2]
    \draw[thick] (0,-.866) arc[start angle=-60, end angle=60, radius=1];
    \draw[thick] (0,.866) arc[start angle=120, end angle=240, radius=1];
    \draw[blue,line width=0.5mm] (-.485,.174) arc[start angle=170, end angle=185, radius=1] to (.406,.423);
    \draw[red,line width=.5mm] (-.451,.309) arc[start angle=162, end angle=145, radius=1] to (.18,.3) to (.406,.423);
    \filldraw (0,-.866) circle (.5pt);
    \filldraw[] (0,.866) circle (.5pt);
    \filldraw[red] (-.406,.423) circle (.5pt);
    \filldraw[blue] (-.485,.174) circle (.5pt);
    \filldraw[blue] (.406,.423) circle (.5pt);
    \filldraw[red] (-.319,.573) circle (.5pt);
    \filldraw[red] (-.451,.309) circle (.5pt);
    \node at (.05,-.95) {$b$};
    \node at (.05,.95) {$t$};
    \node at (-.45,.65) {$\ell_{\widehat{\gamma}_s}$};
    \node at (-.52,.45) {$r$};
    \node at (.52,.45) {$r$};
    \node at (-.58,.3) {$s$};
    \node at (-.62,.15) {$m$};
    \node[blue] at (.0,.1) {$\alpha$};
    \node[red] at (.05,.5) {$\widehat{\gamma}_s$};
    \end{tikzpicture}
    \caption{}
    \label{fig:more_nondegenerate}
\end{figure}

We are finally prepared to prove \Cref{lemm:deg_triang}.
\begin{proof}[Proof of \Cref{lemm:deg_triang}]
Suppose there exists an indecomposable bigon $B\subset X$. As discussed at the beginning of this subsection, we may assume that $B$ satisfies properties~\ref{item:degtriangle1} and~\ref{item:degtriangle2}.

First, we claim that there is a horizontal geodesic $\alpha$ such that $m<l_\alpha<l_{\widehat{\alpha}}$. To this end, choose $l_0 \in (m,a)$ sufficiently large so that, for any $l>l_0$, no geodesic from $L(l)$ to $R(l)$ can intersect $L([0,m])$. As observed in the proof of \Cref{lemm:contra}, there must exist a transverse geodesic $\beta$ with $l_\beta<m$ and $l_0< r_\beta$. Let $\alpha\subset B$ be a horizontal geodesic from $L(r_\beta)$ to $R(r_\beta)$. After replacing a subcurve of $\alpha$ with a subcurve of $\beta$ if needed, we can assume that $r_{\widehat{\alpha}}\geq r_\beta$. Since $\alpha$ cannot intersect $L([0,m])$, we deduce as in the proof of \Cref{lemm:ext} that $\widehat{\alpha}$ must point downward. In particular, it follows that \[l_{\widehat{\alpha}}>r_{\widehat{\alpha}}\geq r_\beta =l_\alpha >m.
\]
This verifies the claim.

\begin{figure} 
    \centering
    \subfloat[]{
    \begin{tikzpicture}[scale=2.2]
    \draw[thick] (0,-.866) arc[start angle=-60, end angle=60, radius=1];
    \draw[thick] (0,.866) arc[start angle=120, end angle=240, radius=1];
    \draw[blue,line width=0.5mm] (-.485,.174) arc[start angle=170, end angle=155, radius=1] to (.485,-.174);
    \filldraw (0,-.866) circle (.5pt);
    \filldraw (0,.866) circle (.5pt);
    \filldraw[blue] (-.485,.174) circle (.5pt);
    \filldraw[] (-.466,-.259) circle (.5pt);
    \filldraw[blue] (-.406,.423) circle (.5pt);
    \filldraw[blue] (.485,-.174) circle (.5pt);
    \node at (.05,-.95) {$b$};
    \node at (.05,.95) {$t$};
    \node at (-.725,.1825) {$L(l_\alpha)$};
    \node at (-.65,.475) {$L(l_{\widehat{\alpha}})$};
    \node at (-.675,-.3) {$L(m)$};
    \node[blue] at (0,.025) {\Large $\alpha$};
    \end{tikzpicture}
    \label{fig:non_degenerate_triangles_a}} \hfill
    \subfloat[]{
    \begin{tikzpicture}[scale=2.2]
    \fill[gray,opacity=.1] (0,.866) arc[start angle=120, end angle=155, radius=1] to (.485,-.174) arc[start angle=-10, end angle=60, radius=1];
    \draw[thick] (0,-.866) arc[start angle=-60, end angle=60, radius=1];
    \draw[thick] (0,.866) arc[start angle=120, end angle=240, radius=1];
    \draw[blue,line width=0.5mm] (0,.866) arc[start angle=120, end angle=155, radius=1] to (.485,-.174);
    \draw[red,line width=.5mm] (0,.866) arc[start angle=60, end angle=-10, radius=1];
    \filldraw (0,-.866) circle (.5pt);
    \filldraw[blue] (0,.866) circle (.5pt);
    \filldraw[blue] (-.406,.423) circle (.5pt);
    \filldraw[blue] (.485,-.174) circle (.5pt);
    \node at (.05,-.95) {$b$};
    \node at (.05,.95) {$t=b'$};
    \node at (-.65,.45) {$L'(m')$};
    \node at (.575,-.2) {$t'$};
    \node[blue] at (-.15,.075) {\Large $L'$};
    \node at (0,.5) {\Large $B'$};
    \node[red] at (.4,.7) {\Large $R'$};
    \end{tikzpicture}
    \label{fig:non_degenerate_triangles_b}} \hfill 
    \subfloat[]{
    \begin{tikzpicture}[scale=2.2]
    \fill[gray,opacity=.1] (0,.866) arc[start angle=120, end angle=155, radius=1] to (.485,-.174) arc[start angle=-10, end angle=60, radius=1];
    \draw[thick] (0,-.866) arc[start angle=-60, end angle=60, radius=1];
    \draw[thick] (0,.866) arc[start angle=120, end angle=240, radius=1];
    \draw[thick] (0,.866) arc[start angle=120, end angle=155, radius=1] to (.485,-.174);
    \draw[red,line width=.5mm] (.4986,.052) to (-.319,.574) arc[start angle=145, end angle=155, radius=1] to (-.256,.322);
    \filldraw (0,-.866) circle (.5pt);
    \filldraw (0,.866) circle (.5pt);
    \filldraw[red] (-.406,.423) circle (.5pt);
    \filldraw[] (.485,-.174) circle (.5pt);
    \filldraw[] (-.466,-.259) circle (.5pt);
    \filldraw[red] (-.256,.322) circle (.5pt);
    \node at (.05,-.95) {$b$};
    \node at (.05,.95) {$t=b'$};
    \node at (-.65,.45) {$L'(m')$};
    \node at (.575,-.2) {$t'$};
    \node[red] at (.05,.45) {\Large $\gamma$};
    \node at (-.675,-.3) {$L(m)$};
    \end{tikzpicture}
    \label{fig:non_degenerate_triangles_c}}
    \caption{}
    \label{fig:non_degenerate_triangles}
\end{figure}
So choose such a geodesic $\alpha$. We may furthermore assume that $\alpha$ is topmost among all curves with the same endpoints. Then $B_b^\alpha$ is decomposable by the vertex perturbation trick. See \Cref{fig:non_degenerate_triangles_a}. Thus $B'= B_t^\alpha$ is indecomposable. Furthermore, as we did in the beginning of this subsection, we may deduce that $B'$ itself satisfies the properties \ref{item:degtriangle1} and \ref{item:degtriangle2} when we choose $L'$, $R'$, $b'$, $t'$, $a'$ and $m'$ appropriately; see \Cref{fig:non_degenerate_triangles_b}.
Note in particular that the restriction of $L'$ to $[m',a']$ equals~$\widehat{\alpha}$. 
Thus, by applying \Cref{lemm:ext} to  $B'$, we may find a geodesic $\gamma\subset B'$ that starts in the interior of $\widehat{\alpha}$, ends in $R'((0,a'))\subset R((0,a))$ and intersects $L'((0,m'))\subset L((m,a))$ in between. However such $\gamma$ cannot exist by \Cref{lemm:contra}. Compare also \Cref{fig:non_degenerate_triangles_c}.
\end{proof}

\bibliographystyle{plain}
\bibliography{bibliography}

\begin{thebibliography}{10}

\bibitem{Ale:48}
A.~D. Aleksandrov.
\newblock Foundations of the inner geometry of surfaces.
\newblock {\em Doklady Akad. Nauk SSSR (N.S.)}, 60:1483--1486, 1948.

\bibitem{AZ:67}
A.~D. Aleksandrov and V.~A. Zalgaller.
\newblock {\em Intrinsic geometry of surfaces}.
\newblock Translated from the Russian by J. M. Danskin. Translations of
  Mathematical Monographs, Vol. 15. American Mathematical Society, Providence,
  R.I., 1967.

\bibitem{AGW:10}
Stephanie Alexander, Mohammad Ghomi, and Jeremy Wong.
\newblock Topology of {R}iemannian submanifolds with prescribed boundary.
\newblock {\em Duke Math. J.}, 152(3):533--565, 2010.

\bibitem{AB:16}
Luigi Ambrosio and J\'{e}r\^{o}me Bertrand.
\newblock On the regularity of {A}lexandrov surfaces with curvature bounded
  below.
\newblock {\em Anal. Geom. Metr. Spaces}, 4(1):282--287, 2016.

\bibitem{BL:10}
Victor Bangert and Yiming Long.
\newblock The existence of two closed geodesics on every {F}insler 2-sphere.
\newblock {\em Math. Ann.}, 346(2):335--366, 2010.

\bibitem{BB:04}
A.~S. Belenkiy and Yu.~D. Burago.
\newblock Bi-{L}ipschitz-equivalent {A}leksandrov surfaces. {I}.
\newblock {\em Algebra i Analiz}, 16(4):24--40, 2004.

\bibitem{BE:00}
M.~Bonk and A.~Eremenko.
\newblock Covering properties of meromorphic functions, negative curvature and
  spherical geometry.
\newblock {\em Ann. of Math. (2)}, 152(2):551--592, 2000.

\bibitem{BK:02}
Mario Bonk and Bruce Kleiner.
\newblock Quasisymmetric parametrizations of two-dimensional metric spheres.
\newblock {\em Invent. Math.}, 150(1):127--183, 2002.

\bibitem{BL:03}
Mario Bonk and Urs Lang.
\newblock Bi-{L}ipschitz parameterization of surfaces.
\newblock {\em Math. Ann.}, 327(1):135--169, 2003.

\bibitem{BM:17}
Mario Bonk and Daniel Meyer.
\newblock {\em Expanding {T}hurston maps}, volume 225 of {\em Mathematical
  Surveys and Monographs}.
\newblock American Mathematical Society, Providence, RI, 2017.

\bibitem{Bry:06}
Robert~L. Bryant.
\newblock Geodesically reversible {F}insler 2-spheres of constant curvature.
\newblock In {\em Inspired by {S}. {S}. {C}hern}, volume~11 of {\em Nankai
  Tracts Math.}, pages 95--111. World Sci. Publ., Hackensack, NJ, 2006.

\bibitem{BI:02}
D.~Burago and S.~Ivanov.
\newblock On asymptotic volume of {F}insler tori, minimal surfaces in normed
  spaces, and symplectic filling volume.
\newblock {\em Ann. of Math. (2)}, 156(3):891--914, 2002.

\bibitem{BBI:01}
Dmitri Burago, Yuri Burago, and Sergei Ivanov.
\newblock {\em A course in metric geometry}, volume~33 of {\em Graduate Studies
  in Mathematics}.
\newblock American Mathematical Society, Providence, RI, 2001.

\bibitem{CR:20}
Paul Creutz and Matthew Romney.
\newblock The branch set of minimal disks in metric spaces.
\newblock 2020.
\newblock preprint arXiv:2008.07413.

\bibitem{Deb:20}
Cl\'{e}ment Debin.
\newblock A compactness theorem for surfaces with bounded integral curvature.
\newblock {\em J. Inst. Math. Jussieu}, 19(2):597--645, 2020.

\bibitem{FS:19}
Fran\c{c}ois Fillastre and Dmitriy Slutskiy.
\newblock Embeddings of non-positively curved compact surfaces in flat
  {L}orentzian manifolds.
\newblock {\em Math. Z.}, 291(1-2):149--178, 2019.

\bibitem{FS:20}
François Fillastre.
\newblock An introduction to {R}eshetnyak's theory of subharmonic distances.
\newblock 2020.
\newblock preprint arXiv:2012.10168.

\bibitem{Kok:14}
Gerasim Kokarev.
\newblock On multiplicity bounds for {S}chr\"{o}dinger eigenvalues on
  {R}iemannian surfaces.
\newblock {\em Anal. PDE}, 7(6):1397--1420, 2014.

\bibitem{LW:18a}
Alexander Lytchak and Stefan Wenger.
\newblock Intrinsic structure of minimal discs in metric spaces.
\newblock {\em Geom. Topol.}, 22(1):591--644, 2018.

\bibitem{LW:20}
Alexander Lytchak and Stefan Wenger.
\newblock Canonical parameterizations of metric disks.
\newblock {\em Duke Math. J.}, 169(4):761--797, 2020.

\bibitem{Mun:00}
James~R. Munkres.
\newblock {\em Topology}.
\newblock Prentice Hall, Inc., Upper Saddle River, NJ, 2000.

\bibitem{New:64}
M.~H.~A. Newman.
\newblock {\em Elements of the topology of plane sets of points}.
\newblock Second edition, reprinted. Cambridge University Press, New York,
  1964.

\bibitem{NR:21}
Dimitrios Ntalampekos and Matthew Romney.
\newblock Polyhedral approximation of metric surfaces and applications to
  uniformization.
\newblock preprint arXiv:2107.07422.

\bibitem{PS:19}
Anton Petrunin and Stephan Stadler.
\newblock Metric-minimizing surfaces revisited.
\newblock {\em Geom. Topol.}, 23(6):3111--3139, 2019.

\bibitem{DPMMS:20}
Guido~De Philippis, Michele Marini, Marco Mazzucchelli, and Stefan Suhr.
\newblock Closed geodesics on reversible {F}insler 2-spheres.
\newblock 2020.
\newblock preprint arXiv:2002.00415.

\bibitem{Raj:17}
Kai Rajala.
\newblock Uniformization of two-dimensional metric surfaces.
\newblock {\em Invent. Math.}, 207(3):1301--1375, 2017.

\bibitem{Res:93}
Yu.~G. Reshetnyak.
\newblock Two-dimensional manifolds of bounded curvature.
\newblock In {\em Geometry, {IV}}, volume~70 of {\em Encyclopaedia Math. Sci.},
  pages 3--163, 245--250. Springer, Berlin, 1993.

\bibitem{Tro:09}
Marc Troyanov.
\newblock Les surfaces \`a courbure int\'egrale born\'ee au sens
  d'{A}lexandrov.
\newblock arXiv:0906.3407v1, 2009.

\bibitem{Wil:49}
Raymond~Louis Wilder.
\newblock {\em Topology of {M}anifolds}.
\newblock American Mathematical Society Colloquium Publications, vol. 32.
  American Mathematical Society, New York, N. Y., 1949.

\bibitem{Zal:56}
V.~A. Zalgaller.
\newblock On the foundations of the theory of two-dimensional manifolds of
  bounded curvature.
\newblock {\em Dokl. Akad. Nauk SSSR (N.S.)}, 108:575--576, 1956.

\end{thebibliography}

\end{document}